\documentclass[12pt, reqno]{amsart}
\usepackage[reqno]{amsmath}
\usepackage{amscd,amsthm,amssymb,graphics}
\usepackage{amsfonts,amssymb,amscd,amsmath,enumitem,verbatim}
\setlist[enumerate,1]{ label=(\arabic*), ref=(\arabic*) }
\usepackage{microtype}
\usepackage[a4paper,top=2cm,left=1.5cm,right=1.5cm]{geometry}
\usepackage{xcolor}
\usepackage{csquotes}
\theoremstyle{plain}
\usepackage{booktabs}
\usepackage{hyperref}
\newcommand{\arXiv}[1]{\href{https://arxiv.org/abs/#1}{arXiv:#1}}

\newtheoremstyle{named}
{3pt}{3pt}
{\itshape}{}
{\bfseries}{.}{.5em}
{#3}
\theoremstyle{named}
\newtheorem*{named*}{}
\theoremstyle{plain}
\newtheorem*{theorem*}{Theorem}
\newtheorem{theorem}{Theorem}[section]
\newtheorem{lemma}[theorem]{Lemma}

\newtheorem{proposition}[theorem]{Proposition}
\newtheorem{conjecture}[theorem]{Conjecture}
\newtheorem{question}[theorem]{Question}
\theoremstyle{definition}
\newtheorem{definition}[theorem]{Definition}
\newtheorem{remark}[theorem]{Remark}
\newtheorem{observation}[theorem]{Observation}
\newtheorem{example}[theorem]{Example}
\newtheorem{algorithm}[theorem]{Algorithm}

\newcommand{\Z}{\mathbb{Z}}
\newcommand{\Q}{\mathbb{Q}}
\newcommand{\R}{\mathbb{R}}
\newcommand{\C}{\mathbb{C}}
\newcommand{\co}{\mathcal{O}}
\newcommand{\OK}[1][K]{\co_{{\:\!\! #1}}}
\newcommand{\OKPlus}[1][K]{\OK[#1]^+}
\newcommand{\U}{\mathcal{U}}

\newcommand{\CC}{\mathcal{C}}
\newcommand{\Ccl}{\CC^{\mathrm{cl}}}
\newcommand{\Cdiag}{\CC^{\mathrm{diag}}}
\newcommand{\Cfr}{\CC^{\mathrm{fr}}}
\newcommand{\cl}{\mathrm{cl}}
\newcommand{\diag}{\mathrm{diag}}

\newcommand{\ve}{\varepsilon}
\newcommand{\vectorstyle}[1]{\mathbf{#1}}
\newcommand{\vv}{\vectorstyle{v}}
\newcommand{\ww}{\vectorstyle{w}}
\newcommand{\ee}{\vectorstyle{e}}
\newcommand{\G}{\mathbf{G}} 
\newcommand{\p}{\mathfrak{p}}
\newcommand{\qf}[1]{\langle #1 \rangle}

\newcommand{\OO}{\OKPlus/\U_K^2}
\DeclareMathOperator{\Tr}{Tr}
\DeclareMathOperator{\NN}{N}
\DeclareMathOperator{\rank}{rank}
\DeclareMathOperator{\truant}{truant}
\newcommand{\Srtruant}{\text{$S_\rho$-\;\!\!}\truant} 
\newcommand{\abs}[1]{\mathopen|#1\mathclose|}
\newcommand{\amax}{\alpha_{\mathrm{max}}}
\newcommand{\qq}[1]{\Q(\!\sqrt{#1})} 
\newcommand{\boundII}{{250\,000}} 
\newcommand{\boundIII}{{250\,000}} 
\newcommand{\boundHuge}{{1\,000\,000}} 

\title{Escalations and criteria over real quadratic fields}

\author{Jakub Kr\'asensk\'y}
\address{Czech Technical University in Prague, Faculty of Information Technology, Department of Applied Mathematics, Thákurova~9, 160~00 Praha~6, Czech Republic}
\email{jakub.krasensky@fit.cvut.cz}

\author{Giuliano Romeo}
\address{Charles University, Faculty of Mathematics and
	Physics, Department of Algebra, Sokolovská 83, 186 00 Praha 8, Czech Republic, \normalfont{and}}
\address{Department of Mathematical Sciences, Politecnico di Torino, Corso Duca degli Abruzzi 24, 10129, Torino, Italy}
\email{romeo.giuliano@matfyz.cuni.cz, giuliano.romeo@polito.it}

\thanks{}


\begin{document}
	\begin{abstract}
		The famous 15-Theorem and 290-Theorem fully characterise universal quadratic forms over $\Q$. Similar theorems exist for every totally real number field, but only over $\qq{5}$ the criterion set is explicitly known. We study criterion sets both theoretically and computationally: We develop the method of escalation over number fields, thus providing a simple proof of finiteness of the criteria and, more importantly, a practical tool for computing them. We illustrate this by explicit computations for $\qq{2}$ and $\qq{3}$, obtaining a conjecture about the corresponding universality criteria; we also prove universality of many escalator lattices. Moreover, we develop the somewhat different theory of escalations for diagonal quadratic forms, obtaining the diagonal criterion set for $\qq{5}$ and conjecturally for $\qq{2}$ and $\qq{3}$.
	\end{abstract}
	
	\maketitle
	
	\section{Introduction}
	The study of quadratic forms with integer coefficients goes back to ancient times before Pythagoras. One of the most natural questions is which integers are represented by a given quadratic form. In 1770, Lagrange proved that the quadratic form $x^2+y^2+z^2+t^2$ represents every positive integer. Nowadays, such a form is called \emph{universal}. Despite the long history, involving even the great Ramanujan \cite{Ra}, the full characterisation of universal quadratic forms over $\Z$ is quite recent -- it was provided by the 15-Theorem of Conway--Schneeberger for classical forms and by the 290-Theorem of Bhargava--Hanke in full generality \cite{Bh,BH}. A \emph{classical} form is a quadratic form that has even coefficients in the cross-terms; it is a natural special case as these forms correspond to integral matrices. 
	
	\begin{named*}[{\bfseries 15-Theorem \normalfont(Conway--Schneeberger)}]
		Let $\varphi$ be a classical positive definite quadratic form over $\Q$. Then:
		\begin{enumerate}
			\item $\varphi$ is universal if and only if $\varphi$ represents all elements of the set $\Ccl_{\Q}=\{1, 2, 3, 5, 6, 7, 10, 14, 15\}$.\label{it:15-1}
			\item Moreover, for every $n \in \Ccl_{\Q}$, there exists a classical positive definite quadratic form over $\Q$ that represents all positive integers except of $n$.\label{it:15-2}
		\end{enumerate}
	\end{named*}
	
	Any set $\CC$ satisfying condition \ref*{it:15-1} is called a \emph{criterion set}. Moreover, condition \ref*{it:15-2} means that $\Ccl_{\Q}$ is the unique minimal criterion set in a very strong sense -- namely, a subset of $\Z^+$ is a criterion set if and only if it contains $\Ccl_{\Q}$. In general, an \emph{$S$-criterion set} is a subset $\CC$ of a given set $S$ such that the representation of all elements in $\CC$ is a sufficient condition for the representation of all elements in $S$. The main focus of the present paper is the study and explicit construction of ($S$-)criterion sets over number fields, with particular attention to quadratic fields. In this setting, a form is \emph{universal} if it represents all totally positive algebraic integers inside a totally real number field $K$, i.e.\ all elements of the ring of integers $\OK$ that are positive with respect to all embeddings of $K$ into $\R$.
	
	Universal quadratic forms can be constructed over all totally real number fields thanks to the asymptotic result of Hsia--Kitaoka--Kneser \cite{HKK}. However, it is typically not known what is the minimal rank of a universal quadratic form over a given number field $K$. Siegel \cite{Si} proved that the only fields where a sum of squares is universal are $K=\Q$, where $4$ squares are enough by Lagrange's four-square theorem, and $K=\qq{5}$, where Maa{\ss} \cite{M} proved that the sum of $3$ squares is universal. However, if arbitrary quadratic forms (rather than just sums of squares) are allowed, the situation becomes much more complicated. On the one hand, there are many results indicating that if a number of variables $n$ is fixed, there will be many fields admitting no universal quadratic form in $n$ variables \cite{BK1, BK2, Ka1, Ka3, KS, KT, KYZ, Man, Ya}. On the other hand, although Kitaoka conjectured that there are only finitely many number fields admitting a universal ternary quadratic form, and his conjecture inspired several papers \cite{CKR, EK, KK, KKK, KTZ, KY, KY-EvenBetter}, the full conjecture is far from resolved. A strong result in this direction is \cite{KKP} (and its more explicit version \cite{KK+}) which shows that there are only finitely many quadratic fields admitting a universal form in $7$ variables; this result is tight, as \cite{Ki2} constructs infinitely many real quadratic fields where a form in $8$ variables is universal (moreover, this form is diagonal, i.e.\ without cross-terms). There are many other relevant results on universal quadratic forms over totally real number fields; we recommend the surveys \cite{Ka-survey, Ki-survey}.
	
	\smallskip
	
	Let us return to criterion sets. First of all, let us explicitly state the other two variants of the 15-Theorem for different classes of quadratic forms. 
	
	\begin{named*}[{\bfseries 290-Theorem \normalfont(Bhargava--Hanke)}]
		The 15-Theorem remains valid if both occurrences of \enquote{classical} are removed and $\Ccl_{\Q}$ is replaced by
		\[
		\begin{aligned} \CC_{\Q} ={}& \{1, 2, 3, 5, 6, 7, 10, 13, 14, 15, 17, 19, 21, 22, 23, 26,\\ & 29, 30, 31, 34, 35, 37, 42, 58, 93, 110, 145, 203, 290\}. \end{aligned} \]
	\end{named*}
	
	\begin{named*}[{\bfseries Diagonal 15-Theorem \normalfont(Conway--Schneeberger)}] 
		The 15-Theorem remains valid if both occurrences of \enquote{classical} are replaced by \enquote{diagonal}. The criterion set does not change: $\Cdiag_\Q=\Ccl_\Q$.
	\end{named*}
	
	The $S$-criterion sets for various $S\subset \Z^+$ have been studied by Bhargava and others. According to \cite[Sec.~6]{Ki-survey} or \cite[p.~674]{Ha}, Bhargava proved the following two statements: 1) A classical quadratic form represents all primes if and only if it represents all of $\Ccl_{\mathbb{P}}= \{\text{primes up to $47$}\}\cup\{67,73\}$.
	2) The same holds for the set of all odd positive integers and $\Ccl_{\mathrm{odd}}=\{1, 3, 5, 7, 11, 15, 33\}$. The corresponding results for not-necessarily-classical quadratic forms are known only conjecturally \cite{Ja,Kap} or under GRH \cite{Ro} in the odd-universality case. In fact, it appears \cite[Sec.~4]{DW} that Bhargava never published a proof of these claims; statement 1) is proven in \cite{Ro} (see also \cite[Thm.~1.5]{BC}), while a proof of 2) for diagonal forms has recently been given in \cite{DW,JKKKO}. For other choices of $S$, see for example \cite{BC,De,DR}.
	
	An interesting and fruitful line of research \cite{EKK, KKO, KLO, Kom, Kom2, Oh} arises when we allow $S$ to consist not of numbers but of quadratic forms of bounded rank. The existence of finite $S$-criterion sets in this situation has been proven by Kim--Kim--Oh \cite{KKO2} over $\Z$ and, more recently, by Chan--Oh \cite{CO} for number fields. However, even over $\Z$, such criteria are usually not unique.
	
	Criterion sets have also been studied in other contexts such as forms over local fields \cite{Be1,Be2,HH,HHX,XZ}, Hermitian forms \cite{KKP0} or sums of polygonal numbers \cite{JuK, KaL}; they do not exist for indefinite quadratic forms \cite{XZ} and for higher degree forms \cite{KP}.
	
	Let $S\subset \OKPlus$ where $\OKPlus$ consists of all totally positive integers in a totally real number field $K$. The existence of a finite $S$-criterion set follows from the more general theorem of Chan--Oh \cite{CO}. However, the only explicit construction has been carried out by Lee \cite{Le}:
	
	\begin{named*}[{\bfseries Norm-45-Theorem \normalfont(Lee)}]
		Let $\varphi$ be a classical positive definite quadratic form over $\qq{5}$. Then $\varphi$ is universal if and only if $\varphi$ represents all elements of the set
		\[
		\Ccl_{\qq{5}}=\Bigl\{1,2,\frac{5+\sqrt5}{2}, \frac{7\pm\sqrt5}{2},2\cdot\frac{5+\sqrt5}{2},3\cdot\frac{5+\sqrt5}{2}\Bigr\}.
		\]
	\end{named*}
	
	The ad hoc name we gave to the theorem is due to the weaker reformulation that $\varphi$ is universal if and only if it represents all totally positive elements with norm at most $45$. Note that in the original paper \cite{Le}, there is no mention of minimality or uniqueness of this set, much less a strong condition like \ref*{it:15-2} from the 15-Theorem. We remedy this in Theorem \ref{th:sq5unique}.
	
	\smallskip
	
	In order to formulate the rest of the Introduction correctly, we switch from the algebraic language of quadratic forms to the more geometric point of view, namely to quadratic lattices. The definitions are to be found in Section \ref{se:prelims}; for now, it is enough to imagine that quadratic lattices are exactly the same as quadratic forms. This simplified point of view is in fact correct in fields of class number $1$ such as $\Q$ or $\qq{D}$ for $D=2,3,5$ that are used in most of our examples.
	
	\smallskip
	
	In \cite{KKR}, Kala and both present authors characterised criterion sets over number fields. They proved that minimal criterion sets are unique and consist exactly of the \emph{critical elements}, i.e.\ all those $\alpha$ for which there exists a lattice representing everything except $\alpha$ (up to multiplication by squares of units); compare this with condition \ref*{it:15-2} from the 15-Theorem.
	
	Beside its theoretical implications, this result allows to decide whether a given element belongs to the criterion set (and thus, in principle, to compute the criterion set up to any given norm), since \cite[Prop.~3.1]{KKR} also provides a simpler characterisation of critical elements: $\alpha$ is critical if and only if there exists a lattice representing all elements of smaller norm, but not $\alpha$ itself. It was mentioned that this can be turned into a finite procedure for computing critical elements, but the details were left unexplained in \cite{KKR}; we remedy this in the present paper in Section \ref{se:finite}. Thus we provide a sound basis for studying criteria computationally.
	
	We then apply the theory developed here and in \cite{KKR} for computing both classical and diagonal criterion sets in $\qq{D}$ for $D=2,3,5$. The aim is to show how to practically construct criterion sets for universality over number fields and to obtain the analogue of the 15-Theorem in this setting.
	
	Section \ref{se:finite} develops the theory of escalations over totally real number fields, providing algorithms for deciding whether a given element is critical and for computation of the criterion set. One by-product is a simple independent proof of the existence of a finite $S$-criterion set for any given $S\subset\OKPlus$ (Theorem \ref{th:finite}, Remark \ref{re:finite}). In Section \ref{se:diagonal}, we analyse escalations for diagonal quadratic forms and show how the theory must be modified in that case. In Section \ref{se:sqrt2}, we construct the classical criterion set for $K=\qq{2}$. According to our Conjecture \ref{co:mainsqrt2}, the minimal classical criterion set over $\qq{2}$ is
	\[
	\Ccl_{\qq{2}} = \bigl\{1, 2+\sqrt2, 3, 3(2+\sqrt2), 3(3\pm\sqrt2)\bigr\}.
	\]
	The actual result is Theorem \ref{th:sqrt2main}: we know that none of the six listed elements can be omitted and that any lattice representing all of them represents all totally positive integers of $\qq{2}$ up to norm $\boundII{}$ (hence, hopefully, all). We also provide a conditional proof of Conjecture \ref{co:mainsqrt2} that relies on the conjectural universality of only $12$ quaternary lattices and on assuming that another quaternary lattice represents all elements except one (see Conjecture \ref{co:13lattices}). This requires significant work; in particular, Proposition \ref{pr:6universality} proves universality of all escalations of $x^2 + (2+\sqrt2)y^2+2z^2$.
	
	Section \ref{se:sqrt3} studies the criterion set over $\qq{3}$; the main result is Theorem \ref{th:sqrt3}, saying that
	\[
	\Ccl_{\qq{3}} \supset \Cdiag_{\qq{3}} \supset \bigl\{1, \ve, 4\pm\sqrt3, \ve(4\pm\sqrt3), 7\pm2\sqrt3, \ve(7\pm2\sqrt3)\bigr\},
	\]
	where $\ve=2+\sqrt{3}$, and the Conjecture \ref{co:sqrt3} that both inclusions hold as equalities. Again, this was checked for all elements up to norm $\boundIII{}$, and proved conditionally, assuming the knowledge of elements represented by one ternary and five quaternary lattices. Further, in Section \ref{se:diagExplicit}, we prove that  
	\[
	\Cdiag_{\qq{5}} = \Bigl\{1,2,\frac{5+\sqrt5}{2}, \frac{7\pm\sqrt5}{2},2\cdot\frac{5+\sqrt5}{2}\Bigr\},
	\]
	and conditionally (assuming the universality of three quaternary diagonal forms) also
	\[
	\Cdiag_{\qq{2}} = \bigl\{1,2+\sqrt2,3,3(2+\sqrt2)\bigr\},
	\]
	see Theorem \ref{th:diagonal}. Finally, Section \ref{se:amax} discusses the curious behaviour of the largest element of the criterion set (see Proposition \ref{pr:non-con}) and compares it with the known results over rational integers.

	\section{Preliminaries and notation} \label{se:prelims}
	
	In this section, we collect the necessary definitions, notation and facts that we shall need throughout the paper. Most of the contents of this section is highly standard (see e.g.\ \cite{OM}), only in the last subsection we repeat the most important results on criterion sets over number fields from \cite{KKR}.
	
	\subsection*{Number fields}
	
	Throughout the paper, $K$ is a totally real number field, i.e.\ a field such that all embeddings $K \hookrightarrow \C$ in fact map $K$ into $\R$, and $\OK$ is its ring of integers. We write $\alpha \succeq \beta$ if $\sigma(\alpha) \geq \sigma(\beta)$ in all embeddings $\sigma$, and $\alpha \succ \beta$ if the inequalities are strict. If $\alpha \succ 0$, we call it \emph{totally positive}; we put $\OKPlus = \{\alpha\in\OK \mid \alpha \succ 0\}$. All nonzero squares are totally positive; in particular, for the set of squares of units, we have $\U_K^2 \subset \OKPlus$.
	
	We use the standard definitions of the \emph{norm} $\NN: K \to \Q$ and \emph{trace} $\Tr: K \to \Q$. Clearly, the norm and the trace of $\alpha\in\OKPlus$ is positive, and, more generally, it follows directly from the definitions that if $\alpha\succeq \beta$ for $\alpha,\beta\in\OKPlus$, then $\NN(\alpha)\geq\NN(\beta)$ and $\Tr(\alpha)\geq\Tr(\beta)$. In particular, given any $\alpha$, there are only finitely many such $\beta$. This is due to the fact that in a given degree, there are only finitely many algebraic integers $\beta \succ 0$ with $\Tr(\beta)$ under a given constant. It is also important to note there are only finitely many elements of $\OO$ with norm under a given constant.
	
	We say that $\alpha \in \OKPlus$ is \emph{indecomposable} if it cannot be written as $\beta+\gamma$ with $\beta,\gamma \in \OKPlus$. We call $\alpha\in\OK$ \emph{squarefree} if it is not divisible by any nontrivial square, i.e.\ if $\alpha/\omega^2 \in \OK$ for $\omega\in\OK$ implies that $\omega$ is a unit. 
	We write $\alpha \equiv \beta \pmod{\delta}$ if $\alpha-\beta \in \delta\OK$.
	If $K$ is a quadratic field, we write $\overline{\alpha}$ for the conjugate of $\alpha$.

	\subsection*{Quadratic forms and lattices}
	
	A \emph{quadratic form} of rank $n$ is a homogeneous polynomial of degree $2$ in $n$ variables. Implicitly we assume all quadratic forms to be integral, i.e.\ to map $\OK^{\;\!n} \to \OK$, and \emph{totally positive definite}, i.e.\ to map every nonzero vector to $\OKPlus$. A quadratic form is called \emph{classical} if all cross-terms are in $2\OK$. In particular, all \emph{diagonal} quadratic forms (those without cross-terms) are classical. We write $\qf{\alpha_1,\ldots,\alpha_n}$ for $\alpha_1x_1^2+\cdots+\alpha_nx_n^2$. We consider the \enquote{empty diagonal form} $\qf{}$, which is a trivial mapping that sends the zero vector to zero, to be a diagonal form in our statements.
	
	However, except for sections on diagonal forms, we prefer another way of looking at quadratic forms: A \emph{quadratic lattice} is a tuple $(L,Q)$ where $L$ is a finitely generated torsion-free $\OK$-module and $Q: L \to \OK$ is a \emph{quadratic map}; this means that $Q(\alpha\vv)=\alpha^2Q(\vv)$ for all $\alpha\in\OK$ and $\vv\in L$, and that the symmetric map $B_Q$ defined by $B_Q(\vv,\ww) = \frac12\bigl(Q(\vv+\ww)-Q(\vv)-Q(\ww)\bigr)$ is bilinear. If $B_Q$ takes only values in $\OK$, we call the quadratic lattice $(L,Q)$ \emph{classical}; generally it can take values in $\frac12\OK$. Note that this is the only place where our terminology differs from that of the most standard source: What we call a quadratic lattice, would be an \enquote{integral quadratic lattice} in O'Meara's book \cite{OM}. A quadratic lattice is \emph{totally positive definite} if $Q(\vv) \succ 0$ for all nonzero $\vv$.
	
	From now on, we usually say just \emph{lattice} instead of \enquote{totally positive definite quadratic lattice}; we shall not encounter any other lattices in this article. Moreover, in Sections \ref{se:sqrt2} and \ref{se:sqrt3}, all lattices are assumed to be classical (but this is of course repeated at the beginning of those sections). We also almost always write just $L$ and not $(L,Q)$ for the lattice; if not specified otherwise, the quadratic map is automatically assumed to be called $Q$.
	
	The \emph{orthogonal sum} of two lattices, written as $L_1\perp L_2$, is the direct sum of the modules $L_1\oplus L_2$ equipped with the map $Q_1+Q_2$, where $Q_i$ is the map corresponding to $L_i$. We also sometimes meet lattices scaled by some $\alpha\in\OKPlus$; this means the lattice $(L,\alpha Q)$. Note that we scale the quadratic map and not the vectors (which would also be possible, but we shall not need it).
	
	If $\sigma: K \to K$ is a field automorphism (in particular the conjugation on a quadratic field), then the \emph{conjugate lattice} of $L$ is $\bigl(\sigma(L),\sigma(Q)\bigr)$, where $\sigma(Q)$ acts by $\vv \mapsto \sigma\bigl(Q(\vv)\bigr)$, and the module $\sigma(L)$ can be constructed as follows: The underlying set is $L$, but equipped with a different structure; addition remains the same, but the new multiplication $\cdot_\sigma$ is defined by $\sigma(\alpha) \cdot_\sigma \vv = \alpha \vv$. Fortunately, we do not need this somewhat involved general definition; we shall only meet conjugates of free lattices, and they can be defined much more easily and naturally, see Observation \ref{ob:conjugate}.
	
	\begin{remark}
		There is a minor mistake in the proof of \cite[Cor.~3.5]{KKR}; instead of $\bigl(L,\sigma(L)\bigr)$, one must use the conjugate lattice $\bigl(\sigma(L),\sigma(Q)\bigr)$ defined above.
	\end{remark}
	
	We say that $L$ \emph{represents} $\alpha\in\OKPlus$ if there is some $\vv\in L$ such that $Q(\vv)=\alpha$; we denote this by $\alpha \to L$. A lattice is \emph{universal} if it represents all elements of $\OKPlus$, and \emph{$S$-universal} for some $S\subset \OKPlus$ if it represents all elements of $S$. Note that for any given $\alpha$, a lattice $L$ either represents all elements of $\alpha\U_K^2$, or none of them; thus, we extend the same terminology ($\alpha$ is represented, $L$ is $S$-universal) for $ \alpha\in\OO$ and $S \subset \OO$. Also, elements of $\OO$ have well-defined norms (not traces!), and, in fact, we often implicitly understand elements of $\OK$ as elements of $\OO$; for example, we might say that a lattice \enquote{represents everything except of $3(5+\sqrt5)/2$}, while in fact the full list of non-represented elements is $3(5+\sqrt5)/2 \cdot \varphi^{2k}$, $k\in \Z$, where $\varphi=(1+\sqrt5)/2$ is the fundamental unit.
	
	Lattices $L_1$ and $L_2$ are \emph{isometric} if there exists a linear bijection between them that preserves the quadratic map; this fact is written as $L_1 \simeq L_2$. Lattices are usually studied up to isometry -- if \enquote{there exists a unique lattice} satisfying some condition, it almost always means a unique isometry class.
	
	The \emph{rank} of a lattice, $\rank(L)$, is the rank of the underlying module, i.e.\ the dimension of the vector space $K \otimes L$. Lattices of rank $1$, $2$, $3$ and $4$ are called \emph{unary}, \emph{binary}, \emph{ternary} and \emph{quaternary}; the only lattice of rank $0$ is the \emph{zero lattice} $\{\vectorstyle{0}\}$. We say that $L$ is \emph{free} if the underlying module is free, i.e.\ if it has a basis; in this case the rank is just the number of basis elements. If the class number of the field, $h(K)$, is $1$, then every lattice is free; this is the case for all three quadratic fields $\qq{D}$ with $D=2,3,5$ which we study in depth in this paper. In general, by the structure theorem for finitely generated modules over Dedekind domains, every lattice with $\rank(L)=r$ can be written as $\mathfrak{a}_1^{-1}\vv_1 + \cdots + \mathfrak{a}_r^{-1}\vv_r$ where $\vv_1, \ldots, \vv_r \in L$ (they form a \emph{pseudobasis}) and all $\mathfrak{a}_i \subset \OK$ are ideals.
	
	There is a natural one-to-one correspondence between quadratic forms and free quadratic lattices. More precisely, this correspondence is between classes of quadratic forms up to equivalence (invertible change of variables) and quadratic lattices up to isometry. A quadratic form $Q$ in $n$ variables can of course be viewed as a quadratic lattice $(\OK^{\;\!n},Q)$; on the other hand, if $(L,Q)$ is a free quadratic lattice and $(\vv_1,\ldots, \vv_r)$ its basis, then $Q(x_1\vv_1+ \cdots + x_r\vv_r)$ is a well-defined quadratic form (a polynomial) in variables $x_1, \ldots, x_r$. The quadratic form and the corresponding free quadratic lattice share many properties, so often there is no need to make a distinction between them. For example, we use the notation $\qf{\alpha_1, \ldots, \alpha_n}$ not only for the diagonal form, but also for the corresponding lattice. The empty diagonal form $\qf{}$ of course corresponds to $\{\vectorstyle{0}\}$. 
	
	If $L_1$ and $L_2$ are lattices satisfying $L_1\subset L_2$, then we say that $L_1$ is a \emph{sublattice} of $L_2$ and that $L_2$ is an \emph{overlattice} of $L_1$ or that it \emph{contains} $L_1$. We do \emph{not} require $\rank(L_1)=\rank(L_2)$ here; however, a crucial fact is that every lattice has only finitely many overlattices of the same rank, i.e.\ spanning the same vector space. (This is due to the fact that classical overlattices are contained in $L^{\#}$, the \enquote{dual} of $L$; arbitrary lattices lie in $\frac12 L^{\#}$.)
	
	For $\vv,\ww \in L$, the Cauchy--Schwarz inequality says that $Q(\vv)Q(\ww) \succeq B_Q(\vv,\ww)^2$. If $\vv_1, \ldots, \vv_n \in L$, then we assign to them their \emph{Gram matrix}; this is an $n \times n$ symmetric matrix $\G$ defined by $\G_{ij} = B_Q(\vv_i,\vv_j)$. The Cauchy--Schwarz inequality generalises to the fact that $\G$ is a \emph{totally positive semidefinite} matrix, i.e.\ $\vectorstyle{x}^{\mathrm{T}}\G \vectorstyle{x} \succeq 0$ for every $\vectorstyle{x}\in K^n$. The vectors $(\vv_1,\ldots,\vv_n)$ are linearly independent if and only if $\G$ is totally positive definite, i.e.\ $\vectorstyle{x}^{\mathrm{T}}\G \vectorstyle{x} \succ 0$ for all nonzero $\vectorstyle{x}$, which happens if and only if $\det \G \neq 0$. In this case the lattice $\OK\vv_1 + \cdots + \OK\vv_n$ generated by these vectors is free of rank $n$; we denote it simply by $\qf{\G}$. For example, the lattice corresponding to $2x^2+2xy+7y^2$ is $\bigl\langle\begin{smallmatrix} 2 & 1 \\ 1 & 7  \\ \end{smallmatrix}\bigr\rangle$.
	
	\begin{observation}\label{ob:conjugate}
		Let $L \simeq \qf{\G}$ and let $\sigma$ be an automorphism of $K$. Then the corresponding conjugate lattice of $L$ is isometric to $\qf{\sigma(\G)}$.
		
		Thus, the conjugate of a quadratic form is just the form where all coefficients have been conjugated.
	\end{observation}

	\subsection*{Local representation, local--global principles}
	
	Finally, let us briefly mention that one can also analogously define lattices over local fields. In particular, an important tool for studying a lattice $L$ over $\OK$ are the \enquote{localised lattices} $\OK[{K_{\p}}] \otimes L$ where $\p \subset \OK$ is a prime ideal and $K_{\p}$ the corresponding completion. The theory of lattices over local fields is well developed, and, in particular, it is easy to determine their sets of represented numbers. If $\alpha \in \OKPlus$ is represented by $\OK[{K_{\p}}] \otimes L$ for all primes $\p$, we say that $\alpha$ is \emph{locally represented} by $L$. If this holds for all $\alpha\in\OKPlus$ (all $\alpha\in S$), then we call $L$ \emph{locally ($S$-)universal}. We say that $L$ \emph{satisfies the local--global principle} if $\alpha \in \OKPlus$ is represented by $L$ if and only if it is locally represented by $L$.
	
	As a black box, we shall use the following: The \emph{class number} of $L$ is the number of isometry classes of lattices that are locally isometric to $L$. This class number is $1$ if and only if $\mathrm{mass}(L) = 1/\abs{\mathrm{Aut}(L)}$, where $\mathrm{Aut}(L)$ is the group of automorphisms of $L$ and the definition of mass is not needed here; the crucial point is that both quantities are easily computable (for example by a single command in Magma), so it is easy to check whether a lattice has class number $1$. This is important because of this well-known fact: A lattice with class number $1$ satisfies the local--global principle.
	
	Therefore, while deciding the set of represented elements of a given lattice might be a very hard open problem, it boils down to a computational task if the class number is $1$. Unfortunately, such lattices are very rare, see \cite{Kirschmer, KirschmerLorch}. In general, we have at least the following result by Hsia--Kitaoka--Kneser \cite{HKK}: Let $\rank(L)\geq 5$. Then there exists a $C$ such if $\alpha\in\OKPlus$ is locally represented by $L$ and $\NN(\alpha) > C$, then $\alpha \to L$. Since there are only finitely many elements of $\OO$ with norm under any bound $C$, we see that every quadratic lattice of rank at least $5$ satisfies the local--global principle with finitely many exceptions. Although the constant $C$ is usually too large to be of any practical value, this asymptotic local--global principle is an extremely important theoretical tool.

	\subsection*{Criterion sets, critical elements}
	
	In this subsection we summarise the necessary notions and results from \cite{KKR}. Note that while the present paper is largely independent (it is enough to understand the few definitions and facts given below), the reader may want to consult \cite{KKR} for more context on criterion sets.
	
	As we explained, the set of elements represented by a lattice is naturally a subset of $\OO$. For our definitions to make sense, it is necessary to formulate them in this language (\enquote{modulo squares of units}), although it looks cumbersome at the first sight. 
	
	\begin{definition}
		Let $S \subset \OO$. We say that $\CC \subset S$ is an \emph{$S$-criterion set} if for every lattice $L$ over $K$, we have the following equivalence: $L$ is $S$-universal if and only if it is $\CC$-universal.
		
		If the equivalence holds for every classical lattice, we speak about an \emph{$(S,\cl)$-criterion set}. If it holds for every diagonal lattice (lattice corresponding to a diagonal quadratic form), we speak about an \emph{$(S,\diag)$-criterion set}. 
	\end{definition}
	
	Often it shall be clear from the context that we only care about classical lattices, and in that case we omit the \enquote{$\cl$}. As with many of the upcoming definitions, if $S=\OO$, we omit the prefix \enquote{$S$-}. In this case, we just say \enquote{criterion set} or sometimes \enquote{universality criterion (set)}.

	\begin{definition}
		Let $L$ be a lattice and $S \subset \OO$. If $\alpha\in S$ is such that $\alpha \not\to L$ but $\beta\to L$ for all $\beta\in S$ with $\NN(\beta) < \NN(\alpha)$, we say that $\alpha$ is an \emph{$S$-truant} of $L$.
	\end{definition}
	
	As before, if $S=\OO$, we just say \enquote{truant}. Some lattices have a unique truant (as an element of $\OO$), but many do not, as there can be several elements of $\OO$ of the same norm. For example, the zero lattice $\{\vectorstyle{0}\}$ has all totally positive units as truants. This non-uniqueness is a problem in some further definitions, and we shall therefore need the following auxiliary notion.
	
	\begin{definition} \label{de:admissible}
		Let $\rho$ be a total order on $S$ such that $\NN(\alpha) > \NN(\beta)$ implies $\alpha >_\rho \beta$. For the purposes of this paper, we call any order with this property \emph{admissible}.
	\end{definition}
	
	The specific choice of $\rho$ is usually not too important.
	
	\begin{definition}
		Let $L$ be a lattice, $S \subset \OO$, and $\rho$ an admissible order on $S$. If $\alpha\in S$ is such that $\alpha \not\to L$ but $\beta\to L$ for all $\beta\in S$, $\beta <_\rho \alpha$, we say that $\alpha$ is \emph{the $S_\rho$-truant} of $L$.
	\end{definition}
	
	Note that every lattice either has a unique $S_\rho$-truant, or is $S$-universal. This allows us to write $\Srtruant(L)$ without any ambiguity.
	
	\begin{definition}
		Let $S \subset \OO$ and $\alpha\in S$. If there exists a lattice $L$ such that $\alpha \not\to L$ but $L$ is $(S\setminus\{\alpha\})$-universal, then we say that $\alpha$ is \emph{$S$-critical}.
		
		Furthermore, if the lattice can be chosen classical, then $\alpha$ is ($S,\cl$)-\emph{critical}. If it can be chosen diagonal (i.e.\ corresponding to a diagonal quadratic form), then $\alpha$ is ($S,\diag$)-\emph{critical}.
		
		We write $\CC_S$ for the set of all $S$-critical elements, $\Ccl_S$ for the set of all ($S,\cl$)-critical elements and $\Cdiag_S$ for the set of all ($S,\diag$)-critical elements. If $S=\OO$, we write simply $\CC_K$, $\Ccl_K$ and $\Cdiag_K$.
	\end{definition}
	
	Typically, it is clear from the context whether the lattices under consideration are classical or even diagonal; in that case it is enough to say $S$-critical. Again, if $S=\OO$, we can just say \enquote{critical}. The reason why we are interested in $S$-critical elements is that each of them obviously must belong to every $S$-criterion set.
	
	It seems very hard to prove that an element is $S$-critical, since determining exactly the set represented by a lattice is often impossible. However, the following crucial lemma from \cite{KKR} shows that in fact, determining whether a given $\alpha$ is critical or not is quite easy.
	
	\begin{proposition}[{\cite[Prop.~3.1]{KKR}}] \label{pr:KKRcritical}
		Let $\alpha \in S$. The following are equivalent:
		\begin{enumerate}[label=(\alph*), ref=(\alph*)]
			\item $\alpha$ is $S$-critical;
			\item there exists a lattice $L$ for which $\alpha$ is an $S$-truant. 
		\end{enumerate}
		Analogous statements hold for ($S,\cl$)-critical and ($S,\diag$)-critical elements.
	\end{proposition}
	
	Finally we are ready to state the main result about the criterion sets. A simplified version is as follows: \emph{For every totally real number field $K$, there exists a unique criterion set $\CC_K$ which is minimal with respect to inclusion. This set $\CC_K$ is finite and consists precisely of the critical elements.}
	
	For the full version, remember that we already defined $\CC_S$ as the collection of all $S$-critical elements, and we know that it is contained in every $S$-criterion set.
	
	\begin{theorem}[{\cite[Thm.~3.3]{KKR}}] \label{th:KKRunique} Let $K$ be a totally real number field and $S \subset \OO$. Then:
		\begin{enumerate}[label=(\alph*), ref=(\alph*)]
			\item $\CC_S$ is an $S$-criterion set. 
			\item $\CC \subset S$ is an $S$-criterion set if and only if $\CC_S \subset \CC$.
			\item $\Cdiag_S \subset \Ccl_S \subset \CC_S$.
			\item $\CC_S$ is finite.
		\end{enumerate}
		The analogous statement holds also for $\Ccl_S$ and $\Cdiag_S$.
	\end{theorem}
	
	Let us stress once again that we often abuse notation and do not make a distinction between elements of $\OKPlus$ and of $\OO$. In particular, we write criterion sets as if they were subsets of $\OKPlus$ (see Theorems \ref{th:sqrt2main} and \ref{th:sqrt3}), while strictly speaking they are subsets of $\OO$. For example, when we write $7+2\sqrt3 \in \Ccl_{\qq{3}}$, one must read it as $(7+2\sqrt3)\U_{\qq{3}}^2 \in \Ccl_{\qq{3}}$. We believe that it increases legibility and does not cause confusion.

	Finally, let us make a few remarks that are specific for the case $S=\OO$. First, the sets $\CC_K$, $\Ccl_K$ and $\Cdiag_K$ are closed under multiplication by totally positive units and under field automorphisms \cite[Cor.~3.5]{KKR}. Second, they only contain squarefree elements, since an element divisible by a nontrivial square can never be critical. Third, the set $\Cdiag_K$ (and thus also the larger sets $\Ccl_K$ and $\CC_K$) contains all squarefree indecomposable elements \cite[Thm.~4.2]{KKR}.

	\section{Finiteness, escalation} \label{se:finite}
	
	Throughout the section, let $K$ be a totally real number field and let $S \subset \OO$ be given. Thanks to Theorem \ref{th:KKRunique} (the main result of \cite{KKR}), we know that there is a unique minimal $S$-criterion set and that it consists precisely of the $S$-critical elements. Now we will prove its finiteness (thus providing a simpler proof for this important special case of \cite[Thm.~5.7]{CO} and hence also for Theorem \ref{th:KKRunique}(d)) and also show a method for computing it explicitly. It is the well-known method of escalation, adapted for number fields. To be more exact, the method of escalation is only capable of computing all elements of the criterion set with norm under a given bound; proving that there are no other critical elements is usually very hard (recall the notoriously hard, still not fully published proof of the 290-Theorem).
	
	\emph{Throughout this section, let the meaning of \enquote{lattice} be fixed -- either as \enquote{totally positive definite quadratic lattice} or as \enquote{classical totally positive definite quadratic lattice}.} The whole theory including proofs works verbatim the same for both cases. On the other hand, the case of diagonal forms behaves somewhat differently, cf.\ Section \ref{se:diagonal}.

	Our aim is to be able to decide whether a given $\alpha \in S$ is $S$-critical. By the crucial Proposition \ref{pr:KKRcritical}, this means to either find a lattice with $S$-truant $\alpha$, or to show that such a lattice cannot exist. This will be achieved by essentially producing all lattices which represent all elements \enquote{smaller} than $\alpha$ and checking whether some of them fails to represent $\alpha$. To formulate this precisely, we need the notion of \emph{admissible order} from Definition \ref{de:admissible}; if not specified otherwise, we assume throughout that $\rho$ is a fixed admissible order on $S$.
	
	\begin{definition} \label{de:escalation}
		Let $L$ be a lattice. Any lattice of the form $L_1=\tilde{L} + \OK \vv$, where $\tilde{L}\simeq L$ and $Q(\vv)$ is an $S$-truant of $L$, is called an \emph{$S$-escalation} of $L$. Informally, $L_1$ was obtained from $L$ by adding a vector which represents the chosen truant. 
		
		Moreover, we say that $L_1$ is an \emph{$S_\rho$-escalation} of $L$ if $Q(\vv)$ is the $S_\rho$-truant of $L$.
		
		Finally, $L$ is an \emph{$S$-escalator} if it can be obtained by finitely many $S$-escalations from the zero lattice $\{\vectorstyle{0}\}$. Analogously we define an \emph{$S_\rho$-escalator}.
	\end{definition}
	
	As before, if $S=\OO$, we usually omit it and write just \emph{escalation} and \emph{escalator}. Observe that if $\Srtruant(L)=\alpha$, then $L \perp \qf{\alpha}$ is always an $S_\rho$-escalation of $L$; hence, a lattice has no $S_\rho$-escalation if and only if it is $S$-universal. 
	
	\begin{example}\label{ex:notincrease}
		Note that if $L_1$ is an escalation of $L$, they can have the same rank. We illustrate it for $K=\Q$ where there is only one choice of $\rho$, namely the standard ordering on $\R$. For example, $\qf{1}$ is an escalation of $\qf{4}$. Less trivially, $\qf{1,1}$ is an escalation of $\qf{2,2}$: If we denote by $\ee_1, \ee_2$ the two orthogonal vectors with $Q(\ee_i)=2$, then $L = \qf{2,2} = \Z \ee_1 + \Z \ee_2$; the vector $\frac12(\ee_1+\ee_2)$, not contained in $L$, represents $1=\truant(L)$; and $L + \Z \cdot \frac12 (\ee_1+\ee_2) = \Z \cdot \frac12(\ee_1+\ee_2) + \Z \cdot\frac12(\ee_1-\ee_2) \simeq \qf{1,1}$. Of course, another possible escalation is $\qf{2,2,1}$, where the rank \emph{has} increased.
		
		In these simple examples, the original lattice was not an escalator. However, the same situation happens even for escalators -- in Example \ref{ex:rankdidntgrow}, we shall see that over $\qq{2}$, some escalations of the escalator $\qf{1,2+\sqrt2,2}$ are isometric to $\qf{1,1,2+\sqrt2}$.
	\end{example}
	
	The key idea behind escalators is that they are the minimal lattices needed for representing a given set of numbers. We express it more precisely in the following proposition. Let us remark that while part \ref*{it:1} is interesting on its own and immediately provides motivation for studying escalators, it shall not be used anywhere; it is the simpler part \ref*{it:2} that we need in our further proofs.
	
	\begin{proposition} \label{pr:subescalators}
		Let $\rho$ be a fixed admissible order and $L$ a lattice. Then:
		\begin{enumerate}
			\item $L$ is $S$-universal if and only if it contains an $S$-universal $S_{\rho}$-escalator. \label{it:1}
			\item If $\Srtruant(L)=\alpha$, then $L$ contains an $S_{\rho}$-escalator $E$ with $\Srtruant(E)=\alpha$. \label{it:2}
		\end{enumerate}
	\end{proposition} 
	\begin{proof}
		Let us start with \ref*{it:2}. We have a lattice $L$ which does not represent $\alpha$ but represents all $\beta\in S$, $\beta <_\rho \alpha$. Put $E_0 = \{\vectorstyle{0}\}$. Then iteratively define $\beta_{i+1} = \Srtruant(E_i)$; if $\beta_{i+1}=\alpha$, then stop and put $E=E_i$, otherwise choose any $\vv_{i+1} \in L$ such that $Q(\vv_{i+1}) = \beta_{i+1}$, and put $E_{i+1} = E_i + \OK \vv_{i+1}$. Since there are only finitely many square classes of elements under a given norm and since $L$ does not represent $\alpha$, the construction eventually stops and we have $\Srtruant(E_i) = \alpha$. All the constructed lattices $E_i$ are $S_\rho$-escalators; in particular, $E$ is an $S_\rho$-escalator as required.
		
		One implication in part \ref*{it:1} is trivial; the proof of the other is very similar to the proof of part \ref*{it:2}. Namely, we define the lattices $E_i$ in the same manner as before, only this time, we stop and put $E=E_i$ only when $E_i$ is $S$-universal. It remains to show that this happens after finitely many steps. The reason is that there cannot exist an infinite growing chain of sublattices in $L$: While it is possible that $E_i$ and $E_{i+1}$ have the same rank for some indices $i$, there are altogether only finitely many proper overlattices of $E_i$ of the same rank, so there is some $j$ such that $\rank(E_j) > \rank (E_i)$; inductively, we would find a lattice $E_k$ of arbitrarily large rank, which is a contradiction, as their rank is bounded by $\rank(L)$. 
	\end{proof}

	The following proposition explains the connection between $S$-escalators, $S_\rho$-escalators, and the $S$-critical elements:
	
	\begin{proposition}\label{pr:escalators}
		For $\alpha \in S$, the following are equivalent:
		\begin{enumerate}
			\item[0)] $\alpha$ is $S$-critical.
			\item[a)] $\alpha$ is an $S$-truant of some lattice.
			\item[b)] $\alpha$ is an $S$-truant of an $S$-escalator.
			\item[c)] For some admissible total order $\rho$, $\alpha$ is the $S_\rho$-truant of some $S_\rho$-escalator.
			\item[d)] For every admissible total order $\rho$, $\alpha$ is the $S_\rho$-truant of some $S_\rho$-escalator.
		\end{enumerate}
	\end{proposition}
	\begin{proof}
		The equivalence between 0) and a) is contained in Proposition \ref{pr:KKRcritical}. All the implications in the direction d) $\Rightarrow$ c) $\Rightarrow$ b) $\Rightarrow$ a) are trivial. It remains to show that 0) $\Rightarrow$ d). Let $L$ be the lattice that represents all elements of $S$ except for $\alpha\U_K^2$. Then, in particular, $\Srtruant(L)=\alpha$ independently of $\rho$, so by Proposition \ref{pr:subescalators}\ref*{it:2}, $L$ contains an $S_\rho$-escalator with $S_\rho$-truant $\alpha$.
	\end{proof}
	
	While condition a) remains the simplest way of proving $S$-criticality, condition d) is very practical for showing that some element is \emph{not} $S$-critical: To understand $S$-critical elements (and thus the set $\CC_S$), it is enough to choose any $\rho$ and to be able to systematically produce all $S_\rho$-escalators. Our aim is to show that there are only finitely many $S_\rho$-escalators altogether (Theorem \ref{th:finite}). We will do this by examining the natural tree structure emerging when escalators are generated: We start with $\{\vectorstyle{0}\}$ and then, iteratively, for every leaf we add one new vertex for each $S_\rho$-escalation. 
	
	First we show the simple part -- in the thus produced tree, every vertex has only finitely many children. Although the principle is well known, we provide a proof for completeness sake.
	
	\begin{lemma}\label{le:finitedegree} 
		Given a lattice $L$, there are only finitely many $S_\rho$-escalations of $L$ (up to isometry).
	\end{lemma}
	\begin{proof}
		If $L$ is $S$-universal, then it has no $S$-escalations at all. Hence, let $\alpha = \Srtruant(L)$. First consider escalations of the same rank as $L$: Since there are altogether only finitely many integral overlattices of $L$ spanning the same quadratic space, there are finitely many such escalations.
		
		So we can assume that the escalation is a lattice $L + \OK\vv$ where $Q(\vv)=\alpha$ and $\vv \notin KL$ -- hence, we know the linear structure, namely $L \oplus \OK\vv$, and it is enough to show that there are finitely many choices for the quadratic map on $L \oplus \OK\vv$. Fix a set of generators $\vv_1, \ldots, \vv_k$ of $L$. Then the map $Q$ is fully given once we know the values $B_Q(\vv_i,\vv)$. Since by Cauchy--Schwarz inequality we have $B_Q(\vv_i,\vv)^2 \preceq Q(\vv_i)\alpha$, there are only finitely many choices for each $B_Q(\vv_i,\vv)$, which proves the claim. 
	\end{proof}
	
	The following lemma is of course most important in the case when $L_0 = \{\vectorstyle{0}\}$ and thus all the lattices are escalators, but it holds in general with the same proof. 
	
	\begin{lemma} \label{le:finitepaths}
		There cannot exist an infinite sequence $(L_i)_{i=0}^\infty$ where each $L_{i+1}$ is an $S_\rho$-escalation of $L_i$.
	\end{lemma}
	\begin{proof}
		Note that with each $S$-escalation, the number of represented elements of $S$ increases (at least by $1$). We aim to show that it is impossible to perform infinitely many consecutive $S$-escalations, since eventually we get an $S$-universal lattice which will make further $S$-escalations impossible.
		
		Realising this, we are in almost the same situation as in the proof of the nontrivial implication in \cite[Prop.~3.1]{KKR} (which is stated here as Proposition \ref{pr:KKRcritical}): There, the aim was to show that a lattice will eventually be $(S \setminus \{\alpha\})$-universal, here it will be $S$-universal. The proof will also be almost the same.
		
		First we need to show that, since no $L_i$ is $S$-universal, we can find an $N$ such that $L_N$ has rank at least $5$. Indeed, since $L_0$ has only finitely many integral overlattices of the same rank, and escalation produces proper overlattices, the rank will eventually be larger than that of $L_0$, and then we again have only finitely many overlattices of the same rank, etc.
		
		We also show that there exists an $N'$ such $L_{N'}$ is locally $S$-universal (and thus, this holds for all later lattices as well). This part of the proof is virtually the same as in the proof of \cite[Prop.~3.1]{KKR}. Thus we omit the details and just provide two pointers for the convenience of the reader: 1) One observes that there are in fact only finitely many places to be handled, then proves the statement for each of these places separately; 2) given a place $\p$, there are only finitely many square classes in $K_\p$, and in each of them, one has to consider the element of $S$ with the minimal $\p$-valuation.
		
		The lattice $L_{\max\{N,N'\}}$ has rank at least $5$ and is locally $S$-universal. By \cite{HKK}, it represents all but finitely many elements of $S$. Thus, after finitely many further $S$-escalations we get a lattice which is $S$-universal.
	\end{proof}
	
	Finally we are ready to prove that the escalation procedure terminates. This is important for the practical computation of $\CC_K$ (see Section \ref{se:sqrt2}), but even more importantly, it proves that the minimal criterion set is finite:
	
	\begin{theorem}\label{th:finite}
		Given any $S \subset \OO$ and an admissible total order $\rho$, there exists an $N$ such that it is impossible to perform more than $N$ consecutive $S_\rho$-escalations starting from $\{\vectorstyle{0}\}$.
		
		In particular, there exist only finitely many $S_\rho$-escalators (up to isometry) and finitely many $S$-critical elements.
	\end{theorem}
	\begin{proof}
		Consider the rooted tree with root $\{\vectorstyle{0}\}$ where the children of any vertex are the isometry classes of its $S_\rho$-escalations. (Here, a technical choice has to be made. We decided to force the graph to be a tree, since this corresponds to how it is practically obtained when one performs the escalation. However, this means that we must allow several vertices to correspond to the same isometry class of lattices, only obtained via a different sequence of escalations -- i.e.\ the vertex \enquote{remembers its parent}. We could also have given up on the tree property and just produce an oriented graph. This would decrease the number of vertices but make the arguments slightly trickier.)
		
		By Lemma \ref{le:finitedegree}, this tree is locally finite (every vertex has finite degree); by Lemma \ref{le:finitepaths}, it contains no infinite paths. Thus, by Kőnig's lemma, it must be a finite tree. This implies both the existence of $N$ (the depth of the tree), and the finite number of $S_\rho$-escalators.
		
		Since there are finitely many $S_\rho$-escalators and each has at most one $S_\rho$-truant (it can be $S$-universal), part d) of Proposition \ref{pr:escalators} shows that there are only finitely many $S$-critical elements.
	\end{proof}
	
	Now, since $S$-critical elements form an $S$-criterion set, we have concluded a separate proof for the existence of a finite $S$-criterion set:
	
	\begin{remark} \label{re:finite}
		The results of this section together with the paper \cite{KKR}, particularly its third section, are entirely self-contained (they only use classical results on lattices, such as the asymptotic local--global principle of Hsia–Kitaoka–Kneser \cite{HKK}) -- in particular, by proving Theorem \ref{th:finite} we removed the necessity of relying on \cite{CO} in the proof of \cite[Thm.~3.3(d)]{KKR}.
		
		Of course, the results of the present paper and of \cite{KKR} do not apply to the more general situation when $S$ contains quadratic lattices; there, one has to rely on \cite{CO}, and criterion sets are typically not unique.
	\end{remark}
	
	If we had an oracle able to decide whether a lattice is $S$-universal, then the following algorithm would always compute the $S$-criterion set in finite time:
	
	\begin{algorithm} \label{al:escalation}
		Input: A set $S$. Output: The minimal criterion set $\CC_S$.
		\begin{enumerate}
			\item Choose an admissible order $\rho$ on $S$.
			\item Start with $\{\text{lattices to be solved}\} = \bigl\{ \{\vectorstyle{0}\} \bigr\}$ and $\CC_S=\{\}$.
			\item Pick a lattice $L$ that has not yet been solved. If there is no such lattice, terminate. \label{it:pick}
			\item If $L$ is $S$-universal, mark it as solved and return to \ref*{it:pick}. Otherwise put $\alpha = \Srtruant(L)$.
			\item Add $\alpha$ into $\CC_S$.
			\item Compute all $S_{\rho}$-escalations of $L$ and add them to $\{\text{lattices to be solved}\}$. Mark $L$ as solved.
			\item Return back to \ref*{it:pick}.
		\end{enumerate}
	\end{algorithm}
	
	Of course, the problem is that the set $\CC_S$ is precisely the oracle which we would need in Algorithm \ref{al:escalation} in order to compute $\CC_S$. In practice, the algorithm can be implemented as follows: Pick a bound on norm $N$. Instead of checking universality, check only that all elements up to norm $N$ are represented. Then the obtained set is a subset of $\CC_S$, and if we manage (by any means) to prove universality of all the escalators that represent all elements with norm under $N$, then we know that we have computed the whole $\CC_S$.
	
	Another practical problem is that there might simply be too many escalators and the computation will be too slow or require too much memory. Even in the case of $\qq{3}$ in Section \ref{se:sqrt3} we shall use quite a lot of tricks so that our computation is not too tedious and does not require pages and pages of computer output.
	
	\begin{remark}
		Formulating escalations just for quadratic forms = free lattices instead of for \emph{all} quadratic lattices would fail, since a sublattice of a free lattice is not necessarily free. Thus, even if the goal is to compute the criterion set $\Cfr_S$ for quadratic forms, it is necessary to implement the escalation procedure for all lattices, as described in this section, and then use the fact that $\Cfr_S=\CC_S$, see \cite[Thm.~5.1]{KKR}.
		
		Non-free escalators can occur because of situations similar to that of Example \ref{ex:notincrease} -- if $E$ is an escalator which is a free lattice, and $E'$ is its escalation such that $\rank E = \rank E'$, there is no reason why $E'$ should be free.
		
		Of course, in our explicit computations in Sections \ref{se:sqrt2}, \ref{se:sqrt3} and \ref{se:diagExplicit}, we use fields with $h(K)=1$, where this remark is irrelevant, as all lattices are free in that case.
	\end{remark}

	\section{Escalation theory for diagonal forms} \label{se:diagonal}
	
	Escalations for diagonal forms are slightly trickier and there is a notable difference to Section \ref{se:finite}. The analogue of the crucial Proposition \ref{pr:escalators} fails to hold if one naively works with \enquote{diagonal escalators} in the sense of \enquote{escalators which happen to be diagonal}. The problem stems from the fact that while every sublattice of a classical lattice is again classical, most sublattices of diagonal lattices are not diagonal.
	
	It would not be necessary to develop the escalation theory for diagonal forms just to prove the finiteness of $\Cdiag_S$; it already follows from the known inclusion $\Cdiag_S\subset \CC_S$ in \cite[Thm.~3.3]{KKR} together with finiteness of $\CC_S$ (previous section or \cite{CO}). But besides being used in the finiteness proof, escalators are useful for computing the criterion set.
	
	Therefore, we present the appropriate analogue for the diagonal setting. We call it \emph{pseudoescalation}, because the definition differs quite significantly from the escalations introduced in Definition \ref{de:escalation}. One of the differences is that while escalations are geometrical in nature, pseudoescalations are purely algebraic.
	
	\begin{definition} \label{de:pseudoescalation}
		Let $D$ be a diagonal form with $S$-truant $\alpha$. Then its \emph{$S$-pseudoescalation} is any quadratic form
		\[
		D' = D \perp \qf{\gamma_1, \ldots, \gamma_n}
		\]
		such that $\alpha\to D'$ but $\alpha \not\to D \perp \qf{\gamma_1,\ldots, \gamma_{i-1},\gamma_{i+1}, \ldots,\gamma_n}$ for all $1 \leq i \leq n$.
		
		Moreover, if $\alpha$ is the $S_\rho$-truant of $D$, then $D'$ is an \emph{$S_\rho$-pseudoescalation} of $D$.
		
		Finally, a diagonal form is an \emph{$S$-pseudoescalator} if it can be obtained by finitely many $S$-pseudoescalations from the zero lattice $\{\vectorstyle{0}\}$ (which we view as the empty diagonal form $\qf{}$). Analogously we define an \emph{$S_\rho$-pseudoescalator}.
	\end{definition}

	Note that while escalation always increases the rank by at most one (and the rank may also stay the same, see Example \ref{ex:notincrease}), pseudoescalation increases the rank by at least one, but possibly much more. For example, over $\Z$, if $S = \{1,7,\ldots \text{(anything $>7$)}\}$, then the form $\qf{1}$ with $S$-truant $7$ has $\qf{1,1,1,1}$ as one of the $S$-pseudoescalations. Fortunately, there are always only finitely many $S$-pseudoescalations (see later in Lemma \ref{le:DIfinitedegree}), since any given element admits only finitely many decompositions as a sum of totally positive integers.
	
	However, in the most important case when $S=\OO$, i.e.\ when we compute $\Cdiag_K$, every pseudoescalation increases the rank by exactly one:
	
	\begin{proposition} \label{pr:diagonalpractical}
		Let $S=\OO$ and let $D$ be a diagonal form with $\Srtruant(D) = \alpha\U_K^2$. Denote $M_{\alpha} = \{\beta \in \OKPlus \mid \beta \preceq \alpha\}$ and $N_{\alpha} = \{\gamma \in \OO \mid \exists \beta \in M_{\alpha} \text{ such that } \beta \to \qf{\gamma} \}$. (In other words, $N_\alpha$ contains all integral elements of the form $\delta^{-2}\beta$ where $\delta \in \OK$ and $\beta \in M_\alpha$.) Then the $S_\rho$-pseudoescalations of $D$ are precisely the lattices $D \perp \qf{\gamma}$ for $\gamma \in N_{\alpha}$.
	\end{proposition}
	Note that $M_\alpha / \U_K^2 \subset N_\alpha$ and sometimes these two sets are the same (e.g.\ if $K=\Q$). Generally, however, one must use $N_\alpha$, as it is not automatically true that $\gamma \preceq \gamma \omega^2$ for $\gamma \in \OKPlus$ and $0 \neq\omega \in \OK$.
	\begin{proof}
		Thanks to the choice of $S$, the form $D$ represents all elements of strictly smaller norm than $\alpha$; in particular, all elements of $M_\alpha \setminus \{\alpha\}$. Thus any form $D \perp L$ represents $\alpha$ if and only if $L$ represents some element of $M_\alpha$.
		
		Hence $D \perp \qf{\gamma}$ for $\gamma \in N_\alpha$ represents $\alpha$ and thus is a valid pseudoescalation. On the other hand, if $D \perp \qf{\gamma_1,\ldots,\gamma_n}$ is a pseudoescalation, then at least one of the $\gamma_i$ must lie in $N_\alpha$; but then the condition of minimality means that $i=n=1$ and the pseudoescalation is of the desired form.
	\end{proof}
	
	So the procedure of computing all pseudoescalators, which is the basis for computing the set $\Cdiag_K$, is as follows: Use any admissible order $\rho$ on $S = \OO$. Start with the empty diagonal form $\qf{}$ and then iteratively find all pseudoescalations of all the already known pseudoescalators -- that is, take any pseudoescalator $D$ and then:
	\begin{enumerate}
		\item Compute $\alpha\U_K^2 = \Srtruant(D)$. (If it does not exist, then $D$ is universal and we can continue with another unsolved pseudoescalator.)
		\item Find all $\beta \preceq \alpha$.
		\item For those $\beta$ which are not squarefree, divide them by all possible nontrivial squares. The collection of all such elements will be $N_\alpha$.
		\item New pseudoescalators will be $D \perp \qf{\gamma}$ for $\gamma \in N_{\alpha}$.
	\end{enumerate}
	
	This is easier to implement and less tedious to compute by hand than the general $S$-pseudoescalation from Definition \ref{de:pseudoescalation} that has to be used if $S \neq \OO$.
	
	\medskip
	
	The principal relation between lattices in Section 3 was simple inclusion. However, for diagonal forms, inclusion does not behave particularly well; note that for example, $\qf{2,2}$ is a sublattice of $\qf{1,1}$. For our theory to work and for computations to run reasonably fast, we shall work with a more restrictive relation:
	
	Just for the sake of this section, let us say that $D_1$ is a \emph{principal diagonal subform} of the diagonal form $D = \qf{\alpha_1,\ldots,\alpha_n}$ if $D_1 = \qf{\alpha_{i_1},\ldots,\alpha_{i_k}}$ where $\{i_1,\ldots,i_k\}$ is a subset of $\{1,\ldots,n\}$.
	
	While the definitions of pseudoescalators and of a principal diagonal subform differ significantly from the definitions of escalators and of a sublattice, the rest of the theory of Section \ref{se:finite} works almost verbatim. We present the statements so that they can be referred to, but we only give the parts of the proofs that differ from their counterparts from previous section.
	
	\begin{proposition} \label{pr:DIsubescalators}
		Let $\rho$ be a fixed admissible order and $D$ a diagonal form. Then:
		\begin{enumerate}
			\item $D$ is $S$-universal if and only if it has an $S$-universal $S_{\rho}$-pseudoescalator as a principal diagonal subform. \label{it:DI1}
			\item If $\Srtruant(D)=\alpha$, then $D$ has an $S_{\rho}$-pseudoescalator $E$ with $\Srtruant(E)=\alpha$ as a principal diagonal subform. \label{it:DI2}
		\end{enumerate}
	\end{proposition} 
	\begin{proof}
		The proof has the same structure and the same main idea as that of Proposition \ref{pr:subescalators}. Put $E_0=\qf{}$ and iteratively define $\beta_{i+1} = \Srtruant(E_i)$. If $E_i$ is $S$-universal (for the proof of \ref*{it:DI1}) or if $\beta_{i+1}=\alpha$ (for the proof of \ref*{it:DI2}), then stop and put $E=E_i$. Otherwise choose $E_{i+1}$ as any principal diagonal subform of $D$ that represents $\beta_{i+1}$ and contains $E_i$ as a principal diagonal subform (we know that such a principal subform exists, as $D$ itself satisfies both), and is minimal with these properties. Then $E_{i+1}$ is an $S_\rho$-escalation of $E_i$, and by induction an $S_{\rho}$-escalator. Since $\rank(E_{i+1}) > \rank(E_i)$ and all the ranks are bounded by $\rank (D)$, the construction must eventually stop and $E$ will be the desired $S_\rho$-pseudoescalator which is a principal diagonal subform of $D$.
	\end{proof}
	
	\begin{proposition}\label{pr:DIescalators}
		For $\alpha \in S$, the following are equivalent:
		\begin{enumerate}
			\item[0)] $\alpha$ is $(S,\diag)$-critical.
			\item[a)] $\alpha$ is an $S$-truant of some diagonal form.
			\item[b)] $\alpha$ is an $S$-truant of an $S$-pseudoescalator.
			\item[c)] For some admissible total order $\rho$, $\alpha$ is the $S_\rho$-truant of some $S_\rho$-pseudoescalator.
			\item[d)] For every admissible total order $\rho$, $\alpha$ is the $S_\rho$-truant of some $S_\rho$-pseudoescalator.
		\end{enumerate}
	\end{proposition}
	\begin{proof}
		The same as for Proposition \ref{pr:escalators}. Only, instead of applying Proposition \ref{pr:subescalators}\ref*{it:2}, one applies Proposition \ref{pr:DIsubescalators}\ref*{it:DI2}.
	\end{proof}
	
	Again, a computation of pseudoescalations generates a tree structure, and again, we can prove its finiteness. 
	
	\begin{lemma}\label{le:DIfinitedegree}
		Every diagonal form $D$ has only finitely many $S_\rho$-pseudoescalations (up to isometry).
	\end{lemma}
	\begin{proof}
		If $D$ is $S$-universal, then it has no $S$-escalations at all. Hence, let $\alpha = \Srtruant(D)$. Now, since there are only finitely many $\beta\in\OKPlus$ such that $\beta \preceq \alpha$ (for example because it implies $\Tr \beta \leq \Tr \alpha$, and there are only finitely many elements of $\OKPlus$ with a given trace), there are only finitely many ways how to express $\alpha = \beta_1+\cdots + \beta_k$ where $k\in\Z^+$ and all $\beta_i\in\OKPlus$. Also, for every $\beta_i$ there are only finitely many unary diagonal forms that represent it, namely all those $\qf{\gamma}$ where $\beta_i/\gamma \in \OK$ is a square. From this it is easily seen that $D$ has only finitely many $S$-pseudoescalations.
	\end{proof}
	
	\begin{lemma} \label{le:DIfinitepaths}
		There cannot exist an infinite sequence $(D_i)_{i=0}^{\infty}$ where each $D_{i+1}$ is an $S_\rho$-pseudoescalation of $D_i$.
	\end{lemma}
	\begin{proof}
		The proof is the same as for Lemma \ref{le:finitepaths}; in fact, some technical details are even somewhat simpler, since we can directly say that $D_5$ has rank at least $5$. Thus, one has to prove that some $D_N$ is locally $S$-universal, and then apply the asymptotic local--global principle of Hsia--Kitaoka--Kneser \cite{HKK} to $D_{\max\{5,N\}}$.
	\end{proof}
	
	Now we prove that the pseudoescalation procedure terminates. As mentioned, we already have two other proofs of finiteness of $\Cdiag_S$ (one using \cite{CO}, the other using Theorem \ref{th:finite}) based on $\Cdiag_S \subset \CC_S$. However, it is important that the pseudoescalation is a finite procedure so that one can implement it for a practical computation of $\Cdiag_S$ (see also our computations in Section \ref{se:diagExplicit} which were performed by hand).
	
	\begin{theorem}\label{th:DIfinite}
		Given any $S \subset \OO$ and an admissible total order $\rho$, there exists an $N$ such that it is impossible to perform more than $N$ consecutive $S_\rho$-pseudoescalations starting from $\qf{}=\{\vectorstyle{0}\}$.
		
		In particular, there exist only finitely many $S_\rho$-pseudoescalators (up to isometry) and finitely many $(S,\mathrm{diag})$-critical elements.
	\end{theorem}
	\begin{proof}
		The same as for Theorem \ref{th:finite}.
	\end{proof}
	
	The computation of $\Cdiag_S$ or decision whether a given $\alpha$ is in $\Cdiag_S$ can be done analogously to the previous section, in particular to Algorithm \ref{al:escalation}. Namely, if we had an oracle able to decide whether a diagonal form is $S$-universal, then the following algorithm would always compute the $S$-criterion set in finite time:
	
	\begin{algorithm} \label{al:DIescalation}
		Input: A set $S$. Output: The diagonal criterion set $\Cdiag_S$.
		\begin{enumerate}
			\item Choose an admissible order $\rho$ on $S$.
			\item Start with $\{\text{forms to be solved}\} = \bigl\{ \qf{} \bigr\}$ and $\Cdiag_S=\{\}$.
			\item Pick a form $D$ that has not yet been solved. If there is no such form, terminate. \label{it:DIpick}
			\item If $D$ is $S$-universal, mark it as solved and return to \ref*{it:DIpick}. Otherwise put $\alpha = \Srtruant(D)$.
			\item Add $\alpha$ into $\Cdiag_S$.
			\item Compute all $S_{\rho}$-pseudoescalations of $D$ and add them to $\{\text{forms to be solved}\}$. (If $S=\OO$, i.e.\ if the aim is to compute $\Cdiag_K$, then use the more explicit description of pseudoescalations from Proposition \ref{pr:diagonalpractical}.) Mark $D$ as solved.
			\item Return back to \ref*{it:DIpick}.
		\end{enumerate}
	\end{algorithm}
	
	In fact, while the assumption \enquote{if we had an oracle able to decide whether a given lattice is $S$-universal} was a rather unrealistic desire in Algorithm \ref{al:escalation}, here we indeed are in such a fortunate situation at least in one specific case: Since Lee found an explicit criterion set for classical forms over $\qq{5}$, we are able to compute $\Cdiag_{\qq{5}}$ in Theorem \ref{th:diagonal}.
	
	We are in a similar situation for certain subsets of $\Z$. Let us illustrate the theory by computing the diagonal criterion for odd-universality.
	
	\begin{example}
		Using the knowledge $\Ccl_{\mathrm{odd}}=\{1,3,5,7,11,15,33\}$, see \cite[Thm.~1.5]{BC}, we shall prove that $\Cdiag_{\mathrm{odd}}=\Ccl_{\mathrm{odd}}$.
		
		In fact, for a clean proof of this equality, it is enough to exhibit a diagonal form with truant $\alpha$ for every $\alpha\in\Ccl_{\mathrm{odd}}$, so let us start with this. One easily checks that $\mathrm{odd}\text{-}\!\truant(\qf{})=1$, $\mathrm{odd}\text{-}\!\truant(\qf{1})=3$, $\mathrm{odd}\text{-}\!\truant(\qf{1,2})=5$, $\mathrm{odd}\text{-}\!\truant(\qf{1,1,1})=7$, $\mathrm{odd}\text{-}\!\truant(\qf{1,3,5})=11$, $\mathrm{odd}\text{-}\!\truant(\qf{1,3,4})=15$, $\mathrm{odd}\text{-}\!\truant(\qf{1,3,4,15})=33$. This concludes the proof -- we just proved $\Cdiag_{\mathrm{odd}}\subset\Cdiag_{\mathrm{odd}}$, and the other inclusion holds always.
		
		However, we shall use this opportunity to illustrate the pseudoescalation procedure. The odd-truant of $\qf{}$ is $1$, and clearly, its only pseudoescalation is $\qf{1}$. The odd-truant of $\qf{1}$ is $3$, and we claim that the pseudoescalations are $\qf{1,1,1}$, $\qf{1,2}$ and $\qf{1,3}$. Indeed, for $\qf{1, \gamma_1, \ldots, \gamma_n}$ to represent $3$, at least one $\gamma_i$ must be $1$, $2$ or $3$; and we easily see that these three are the minimal forms with this property.
		
		The pseudoescalator $\qf{1,1,1}$ already represents all of $\Ccl_{\mathrm{odd}}$ except for $7$ and $15$. Its odd-truant is $7$, and this is in fact smaller than any non-represented even integer, hence pseudoescalations are $\qf{1,1,1,n}$ for $1\leq n \leq 7$. (In fact, Proposition \ref{pr:diagonalpractical} can be applied whenever $S$-truant is the same as the \enquote{full} truant.) All of them represent $15$, so they are classical lattices representing all of $\Ccl_{\mathrm{odd}}$, hence they are odd-universal. (One can also observe that they are in fact universal by the 15-Theorem.)
		
		We skip most of the branch with the pseudoescalator $\qf{1,2}$, since nothing particularly interesting happens there. We shall look only at one of its pseudoescalations, namely $\qf{1,2,5}$. Its odd-truant is $15$, and it also represents 1–9 and 11–14. Thus the fourteen diagonal forms $\qf{1,2,5,n}$ for $n \neq 5$, $n \leq 15$, are clearly its pseudoescalations, as $(15-n) \to \qf{1,2,5}$. One easily sees that the only other pseudoescalation is $\qf{1,2,5,5,5}$; again, the rank increased by $2$.
		
		Let us switch to the last branch. The odd-truant of $\qf{1,3}$ is $5$. This time, the only pseudoescalations are $\qf{1,3,n}$ for $n=1,2,4,5$; there is no pseudoescalation containing $\qf{1,3,3}$, since none of the forms $\qf{1,3,3,3,\ldots}$ represents $5$.
		
		Again, we skip most of the computations, and only look at pseudoescalations of $\qf{1,3,4}$. These are slightly trickier to compute, since $\qf{1,3,4}$ fails to represent many numbers smaller than its odd-truant $15$, namely $2$, $6$, $10$ and $14$. Assume now that $\qf{1,3,4,\gamma_1, \ldots, \gamma_k}$ is a pseudoescalator. From the knowledge of numbers represented by $\qf{1,3,4}$, we see that for the whole form to represent $15$, it is necessary and sufficient that $\qf{\gamma_1,\ldots, \gamma_k}$ represents at least one $m \leq 15$ other than $1, 5, 9, 13$. Thus $\qf{1,3,4,n}$ is a pseudoescalation if $n\leq 15$, $n \neq 1,5,9,13$. In fact, $\qf{1,3,4,1}$ is also a pseudoescalation since $4 \to \qf{1}$. Every other pseudoescalation is of the form $\qf{1,3,4,\gamma_1, \ldots, \gamma_k}$ where all $\gamma_i$ are in $\{5,9,13\}$.  Thus one can observe that the only such pseudoescalations are $\qf{1,3,4,5,5}$ and $\qf{1,3,4,5,9}$.
	\end{example}
	
	Beside the difference in the definitions, there is only one other statement from Section \ref{se:finite} that does not generalise straightforwardly for diagonal forms. It is the following theorem of \cite{KKR}, which claims that there is no difference between the criterion sets for forms and for lattices.
	
	\begin{theorem*}[{\cite[Thm.~5.1]{KKR}}]
		Let $\CC \subset S$. Then $\CC$ is an $S$-criterion set if and only if the following equivalence holds for every \textbf{free} quadratic lattice $L$: \enquote{$L$ is $S$-universal if and only if $L$ is $\CC$-universal}. The analogous statement holds for free classical lattices and the $(S,\cl)$-criterion set. 
	\end{theorem*}
	
	This fails to carry over to the diagonal setting simply because it is not clear what should be considered a non-free diagonal lattice. One might consider lattices with an orthogonal pseudobasis, i.e.\ lattices that can be written as an orthogonal sum of unary lattices; however, then we cannot prove an analogue of the theorem above. The problem is as follows: A lattice with an orthogonal pseudobasis which happens to be free can of course be interpreted as a quadratic form, but not automatically as a diagonal one.
	
	On the other hand, if lattices that decompose as an orthogonal sum of unary lattices are considered important, one can probably reprove all results of the present section in this setting and obtain uniqueness of the minimal criterion set for $S$-universality of such lattices (using a notion like \enquote{principal diagonal sublattice that has an orthogonal pseudobasis}); but we don't see any reason why this criterion set should be equal to $\Cdiag_S$.

	\section{Universality criterion set for \texorpdfstring{$\qq{2}$}{Q(√2)}} \label{se:sqrt2}
	
	In this section, we illustrate the developed theory by computing the universality criterion set $\Ccl_{\qq{2}}$. The result will be conjectural; what we \emph{know} is that we found all critical elements with norm below $\boundII{}$. We do so by computing (probably) all escalator lattices over this field.
	
	Since we compute $\Ccl_{\qq{2}}$, \emph{throughout this section, all lattices are assumed to be classical}. The computation of $\CC_{\qq{2}}$ would be in principle the same, just computationally much more involved. Note that Lee's result \cite{Le} for $\qq{5}$ also requires the \enquote{classical} assumption.

	\begin{conjecture} \label{co:mainsqrt2}
		For the minimal classical criterion set over $\qq{2}$, we have
		\[
		\Ccl_{\qq{2}} = \bigl\{1, 2+\sqrt2, 3, 3(2+\sqrt2), 3(3\pm\sqrt2)\bigr\}.
		\]
		The smallest rank of a universal escalator is $3$, the largest is $5$; any lattice obtained from $\{\vectorstyle{0}\}$ by five escalations is universal.
	\end{conjecture}
	
	We also directly state how much we are able to prove. For the sake of completeness, we provide the proof immediately, although it essentially consists of a list of references to statements proved throughout this section.
	
	\begin{theorem} \label{th:sqrt2main}
		Let $\CC = \bigl\{1, 2+\sqrt2, 3, 3(2+\sqrt2), 3(3\pm\sqrt2)\bigr\}$. Then:
		\begin{enumerate}
			\item $\CC \subset \Ccl_{\qq{2}}$. \label{it:first}
			\item $\Ccl_{\qq{2}}$ contains no other element with norm below $\boundII{}$.
			\item $\CC$ is the $S$-criterion set for $S = \{\alpha \in \Z[\sqrt2]^{+} \mid \NN(\alpha) \leq \boundII{}\}$. \label{it:third}
			\item Conjecture \ref{co:mainsqrt2} holds if Conjecture \ref{co:13lattices} does, i.e.\ if certain 12 quaternary lattices are indeed universal and one other quaternary lattice represents everything except its truant $3(3-\sqrt2)$. \label{it:last}
		\end{enumerate}
	\end{theorem}
	
	\begin{proof}
		Parts \ref*{it:first}--\ref*{it:third} follow from Propositions \ref{pr:ternaryesc}, \ref{pr:6summary} and \ref{pr:5summary}. Since their proof requires a considerable effort including the use of a computer program, we provide here an alternative proof of \ref*{it:first} that can be easily checked by hand. It is just a condensed version of certain parts of proofs of the listed propositions.
		
		Take $\alpha \in \CC$. By Proposition \ref{pr:KKRcritical}, it is enough to find a (classical) lattice with truant $\alpha$. Put $E_0\simeq\{\vectorstyle{0}\}$, $E_1\simeq\qf{1}$, $E_2\simeq\qf{1,2+\sqrt2}$, $E_3^{(5)}\simeq\qf{1,2+\sqrt2,3}$, $E_{4,(5)}^{(2)}\simeq\qf{1,2+\sqrt2} \perp\bigl\langle\begin{smallmatrix}
			3 & 1 \\
			1 & 3(2+\sqrt2)  \\
		\end{smallmatrix}\bigr\rangle$ and $E_{4,(5)}^{(3)}\simeq\qf{1,2-\sqrt2} \perp\bigl\langle\begin{smallmatrix}
			3 & 1 \\
			1 & 3(2-\sqrt2)  \\
		\end{smallmatrix}\bigr\rangle$. Then these lattices have truants $1$, $2+\sqrt2$, $3$, $3(2+\sqrt2)$, $3(3-\sqrt2)$ and $3(3+\sqrt2)$ in this order.
		
		Finally, the proof of part \ref*{it:last} requires almost all statements from this section. We need to show that, under the assumption, all escalators which do not have truant in $\CC$ are universal. For ternary lattices, this is contained in Proposition \ref{pr:ternaryesc}; for escalations of $\qf{1,2+\sqrt2,2}$, we proved it unconditionally in Proposition \ref{pr:6universality}. For escalations of $\qf{1,2+\sqrt2,3}$, we performed some reduction; the necessary implication is contained in Lemma \ref{le:conjecturesuffices}.
	\end{proof}
	
	Let us repeat some basic facts about $\qq{2}$. Its class number is $1$, which means that all lattices are free and we could reformulate everything in terms of quadratic forms, if we wished. But since escalation is based on work with vectors, it is more convenient to stay in the world of lattices. We use $\overline{\alpha}$ for the conjugate of $\alpha$. The ring of integers is $\Z[\sqrt2]$; the fundamental unit $\ve = 1+\sqrt2$ is not totally positive. Up to multiplication by $\varepsilon^{2k}$, the only indecomposable elements are $1$ and $2+\sqrt2$; observe that according to Theorem \ref{th:sqrt2main} they both lie in $\Ccl_{\qq{2}}$, just as \cite[Thm.~4.2]{KKR} predicts. In order to shorten some long lists, we shall occasionally write $\pi_2=2+\sqrt2$.
	
	Remember that we work with elements modulo squares of units, so we do not distinguish between $2+\sqrt2$ and $2-\sqrt2 = (2+\sqrt2)\ve^{-2}$. Usually, when performing escalation, one needs to choose an admissible order $\rho$, see Definition \ref{de:admissible}. However, it turns out that over $\qq{2}$, every lattice has a unique truant independent of the ordering.
	
	We start with the escalator $E_0 \simeq \{\vectorstyle{0}\}$ of rank zero, with $\truant(E_0)=1$. Its only escalation is $E_1 \simeq \qf{1}$ of rank one, with $\truant(E_1)=2+\sqrt2$. 
	
	Next, we need to find all lattices of the form $E_1 + \OK \vv$ where $\vv$ represents $2+\sqrt2$. The Gram matrix $\G_2$ of $(\ee_1,\vv)$, where $\ee_1$ is the basis vector of $E_1$, is
	\[
	\G_2 = 
	\begin{pmatrix}
		1 & a \\
		a & 2+\sqrt2 \\
	\end{pmatrix} \text{ for some $a \in \Z[\sqrt2]$};
	\]
	this matrix must be totally positive semidefinite, which is equivalent to $a^2\preceq 2+\sqrt2$. Since $2+\sqrt2$ is indecomposable and not a square, the only choice is $a=0$. Thus there is only one escalator of rank 2, namely $E_2 \simeq \langle 1,2+\sqrt2 \rangle$. One easily sees that $2$, as well as both elements of norm $7$, namely $3\pm\sqrt2$, are represented by $E_2$, but $3$ is not. Hence $\truant(E_2)=3$.
	
	The next escalation step concerns lattices of the form $E_2 + \OK\vv$ with $Q(\vv)=3$. We denote by $\ee_1, \ee_2$ the basis vectors of $E_2$ representing $1$ and $2+\sqrt2$. The Gram matrix $\G_3$ of $(\ee_1,\ee_2,\vv)$ is
	\[
	\G_3 = 
	\begin{pmatrix}
		1 & 0 & a  \\
		0 & 2+\sqrt2 & b \\
		a & b & 3
	\end{pmatrix} \text{ for some $a,b \in \Z[\sqrt2]$};
	\]
	this can be simplified by instead looking at the Gram matrix $\G'_3$ of $(\ee_1,\ee_2,\vv-a\ee_1)$ which of course generates the same quadratic lattice:
	\[
	\G'_3 =
	\begin{pmatrix}
		1 & 0 & 0  \\
		0 & 2+\sqrt2 & b \\
		0 & b & 3-a^2\\
	\end{pmatrix}
	\]
	So we see that every escalation of $E_2$ is of the form $\qf{1} \perp L$ where $L$ is given by the Gram matrix $\bigl(\begin{smallmatrix}
		2+\sqrt2 & b \\
		b & 3-a^2  \\
	\end{smallmatrix}\bigr)$, where the inequalities $a^2 \preceq 3$ and $b^2\preceq (3-a^2)(2+\sqrt2)$ must be satisfied.
	
	Notice that $a^2 \in \{0,1,2\}$, and from the shape of $\G'_3$ it is easy to see that two matrices with the same values of $a^2$ and $b$ are equal. Moreover, matrices which differ only by the sign of $b$ correspond to isometric lattices. Bearing this in mind, the inequalities have seven relevant solutions, namely
	\[
	(a,b)=(\sqrt{2},0),(1,\ve),(0,1),(0,\ve),(0,0),(1,0),(1,1).
	\]
	This yields the following lattices:
	\begin{alignat*}{3}
		&E_3^{(1)}\simeq\langle 1,1,2+\sqrt2\rangle,\quad 
		&&E_3^{(2)}\simeq\langle 1 \rangle\perp\Bigl\langle\begin{matrix}
			2+\sqrt2 & 1 \\
			1 & 2-\sqrt2  \\
		\end{matrix}\Bigr\rangle,\quad
		&&E_3^{(3)}\simeq\langle 1 \rangle\perp\Bigl\langle\begin{matrix}
			2+\sqrt2 & 1 \\
			1 & 3  \\
		\end{matrix}\Bigr\rangle,
		\\
		&E_3^{(4)}\simeq\langle 1 \rangle\perp\Bigl\langle\begin{matrix}
			2-\sqrt2 & 1 \\
			1 & 3  \\
		\end{matrix}\Bigr\rangle,\quad
		&&E_3^{(5)}\simeq\langle 1,2+\sqrt2,3\rangle,\quad
		&&E_3^{(6)}\simeq\langle 1,2+\sqrt2,2\rangle,
	\end{alignat*}
	which correspond (in the second and fourth case after a slight change of basis) to the first six pairs. 
	We did not list the lattice $\qf{1}  \perp\bigl\langle\begin{smallmatrix}
		2+\sqrt2 & 1 \\
		1 & 2  \\
	\end{smallmatrix}\bigr\rangle$ corresponding to the pair $(1,1)$, since performing a suitable basis change shows that it is in fact isometric to $E_3^{(2)}$. For the first time, some of the obtained escalators are universal:
	
	\begin{proposition} \label{pr:ternaryesc}
		Up to isometry, there are six ternary escalators over $\qq{2}$: Four universal ternary escalators, namely $E_3^{(i)}$ for $i \leq 4$, and two ternary escalators with truant $3(2+\sqrt2)$, namely $E_3^{(5)}\simeq\langle 1,2+\sqrt2,3\rangle$ and $E_3^{(6)}\simeq\langle 1,2+\sqrt2,2\rangle$. 
	\end{proposition}
	\begin{proof}
		We have already computed that these lattices are the only escalations of $E_2$ up to isometry. It can easily be checked that there are no further isometries between them. Computing the truants of the two non-universal lattices is easy. The first four lattices correspond to the quadratic forms $x^2+y^2+(2+\sqrt2)z^2$, $x^2 + (2+\sqrt2)y^2 + 2yz + (2-\sqrt2)z^2$ and $x^2 + (2\pm \sqrt2)y^2 + 2yz + 3z^2$, which are universal by \cite{CKR}.
	\end{proof}
	
	Strictly speaking, at this stage we cannot know that there are no other ternary escalators; it should be borne in mind that the escalation does not necessarily increase the rank (although it always happened so far), as showed in Example \ref{ex:notincrease}. In fact, we shall see in Example \ref{ex:rankdidntgrow} that escalation of $E_3^{(6)}$ can yield a ternary lattice. However, it will be one of the already known lattices, so the above proposition is indeed correct.
	
	For the non-universal lattices $E_3^{(5)}$ and $E_3^{(6)}$, we must perform another escalation. Let us note that the class number of $E_3^{(6)}$ is $1$ and so we fully understand its set of represented elements, see Lemma \ref{le:diagsublat}\ref*{it:diagsublat1}. The lattice $E_3^{(5)}$ has class number $2$, hence the local--global principle can fail, and we shall see in Remark \ref{re:l-gfails} that it does. This is the main reason why we are in the end not able to prove that we found the whole $\Ccl_{\qq{2}}$. Lee \cite{Le} succeeded in proving the analogous statement for $\qq{5}$, namely the Norm-45-Theorem in the Introduction, because all his ternary escalators had class number $1$.
	
	We shall first compute the escalations of the better-behaved lattice $E_3^{(6)}\simeq\langle 1,2+\sqrt2,2\rangle$. As before, we look at all possible Gram matrices $\G_{4,(6)}$ of $(\ee_1, \ee_2, \ee_3, \vv)$ where $\vv$ represents the previous truant $3(2+\sqrt2)$ and the $\ee_i$'s form the basis of $E_3^{(6)}$. We have
	\begin{equation} \label{eq:G46}
		\G_{4,(6)} = \begin{pmatrix}
			1 & 0 & 0 & a\\
			0 & 2+\sqrt2 & 0 & b \\
			0 & 0 & 2 & c\\
			a & b & c & 3(2+\sqrt2)\\
		\end{pmatrix},
	\end{equation}
	which directly yields the conditions $a^2 \preceq 3(2+\sqrt2)$, $b^2 \preceq 3(2+\sqrt2)^2$ and $c^2 \preceq 6(2+\sqrt2)$. These are necessary but not sufficient for $\G_{4,(6)}$ to be totally positive semidefinite; they give finitely many solutions which then must be tested for semidefiniteness individually. The diagonal shape of the upper left $3 \times 3$ submatrix allows us to freely switch signs of the numbers $a$, $b$, $c$ without changing the isometry class of the corresponding lattice. Bearing this in mind and performing the necessary computations, we arrive at $38$ triples $(a,b,c)$ that yield positive definite matrices, namely\footnote{Here and in all later lists of this kind (namely the list \eqref{eq:list5} and the list in the proof of Proposition \ref{pr:6universality}), every row (except for the last one) contains exactly 10 entries, simplifying orientation while keeping compactness. We also write $\pi_2=2+\sqrt2$ as it takes up less space.}
	\newcommand{\sep}{,\,}
	\begin{alignat}{10}\label{eq:list6}
		&(0,0,0)\sep &&(0,0,1)\sep &&(0,0,\ve)\sep &&(0,0,\!\sqrt2)\sep &&(0,0,\pi_2)\sep &&(0,1,0)\sep &&(0,\ve,0)\sep &&(0,\ve,1)\sep &&(0,\ve,\ve)\sep &&(0,\ve,\!\sqrt2)\sep \nonumber\\ &(0,\ve,\pi_2)\sep &&(0,\ve^2,0)\sep &&(0,\pi_2,0)\sep &&(0,\pi_2,1)\sep &&(0,\pi_2,\ve)\sep &&(0,\pi_2,\!\sqrt2)\sep &&(0,\pi_2,\pi_2)\sep &&(0,2\ve,0)\sep &&(0,2\ve,1)\sep &&(0,2\ve,\ve)\sep \nonumber\\ &(1,0,0)\sep &&(1,0,1)\sep &&(1,0,\ve)\sep &&(1,0,\pi_2)\sep &&(1,\ve,0)\sep &&(1,\ve,\ve)\sep &&(1,\ve,\pi_2)\sep &&(1,\pi_2,0)\sep &&(1,\pi_2,\ve)\sep &&(\ve,0,0)\sep \nonumber\\ &(\ve,0,1)\sep &&(\ve,0,\ve)\sep &&(\ve,0,\!\sqrt2)\sep &&(\ve,\ve,0)\sep &&(\ve,\ve,1)\sep &&(\ve,\ve,\!\sqrt2)\sep &&(\ve,\pi_2,0)\sep &&(\ve,\pi_2,1).
	\end{alignat}
	and two more triples that yield positive semidefinite but not definite matrices, namely
	\[
	(1,\pi_2,\pi_2), \quad (\ve,\pi_2,\!\sqrt2).
	\]
	
	We shall denote the lattices corresponding to matrices from the first list by $E_{4,(6)}^{(i)}$ for $1 \leq i \leq 38$. While in all previous steps, all obtained matrices were totally positive definite, now we also got two matrices with determinant zero -- they are semidefinite, not definite. This means that they correspond to lattices of rank less than $4$. We show that they are in fact just copies of $E_3^{(1)} = \qf{1,1,2+\sqrt2}$. Although this might seem counter-intuitive, one should have seen it coming: Since $2 = \sqrt2^2$, $E_3^{(6)}$ is a sublattice of $E_3^{(1)}$, which is universal and thus of course represents $3(2+\sqrt2)$.
	
	\begin{example} \label{ex:rankdidntgrow}
		Among the escalations of the ternary lattice $E_3^{(6)}$, there is a universal ternary lattice isometric to $\qf{1,1,2+\sqrt2} \simeq E_3^{(1)}$.
	\end{example}
	\begin{proof}
		We already computed that among the escalations, there is a lattice of the form $\OK\ee_1 + \OK\ee_2 + \OK\ee_3 + \OK\vv$ where the Gram matrix of $(\ee_1,\ee_2,\ee_3,\vv)$ is as in \eqref{eq:G46}
		for $(a,b,c) = (1+\sqrt2,2+\sqrt2,\sqrt2)$. (Or $(1,2+\sqrt2,2+\sqrt2)$, which after a similar computation leads to another proof of our statement.) 
		
		The vectors $(\ee_1,\ee_2,\ee_3,\vv)$ generate the same lattice as $(\ee_1,\vv-\ve\ee_1-\ee_2,\ee_2,\ee_3)$; the corresponding Gram matrix is now
		\[
		\Biggl(\begin{smallmatrix}
			1 & 0 & 0 & 0\\
			0 & 1 & 0 & \sqrt2 \\
			0 & 0 & 2+\sqrt2 & 0\\
			0 & \sqrt2 & 0 & 2\\
		\end{smallmatrix}\Biggr).
		\]
		Changing the set of generators still further to $\bigl(\ee_1,\vv-\ve\ee_1-\ee_2,\ee_2,\ee_3 - \sqrt2(\vv-\ve\ee_1-\ee_2)\bigr)$, which now means only a change of the last vector, gives the Gram matrix
		\[
		\Biggl(\begin{smallmatrix}
			1 & 0 & 0 & 0\\
			0 & 1 & 0 & 0 \\
			0 & 0 & 2+\sqrt2 & 0\\
			0 & 0 & 0 & 0\\
		\end{smallmatrix}\Biggr).
		\]
		Since we assume all lattices to be totally positive definite, the equality $Q\bigl(\ee_3 - \sqrt2(\vv-\ve\ee_1-\ee_2)\bigr)=0$ means that this is in fact the zero vector. So the vectors $(\ee_1,\ee_2,\ee_3,\vv)$ generate the same lattice as $(\ee_1,\vv-\ve\ee_1-\ee_2,\ee_2)$, and the Gram matrix above shows that it is isometric to $\qf{1,1,2+\sqrt2}$.
	\end{proof}
	
	Thus, up to isometry there is one ternary escalation of $E_3^{(6)}$ and it is the already known escalator $E_3^{(1)}$. We are able to prove that all the remaining, quaternary escalations of $E_3^{(6)}$ are also universal, but we postpone it until later (in order to first present the whole escalation process). We will now present the concluding result for all escalations of $E_3^{(6)}$.
	
	\begin{proposition} \label{pr:6summary}
		Up to isometry, there are 1 ternary and 30 quaternary escalations of $E_3^{(6)}$; they are all universal. The ternary escalation is isometric to $\qf{1,1,2+\sqrt2}$, which already occurred as an escalation of $E_2$.
		The quaternary escalations are
		\[
		\qf{1} \perp \Bigl\langle
		\begin{smallmatrix}
			2+\sqrt2 & 0 & b \\
			0 & 2 & c\\
			b & c & 3(2+\sqrt2)-a^2
		\end{smallmatrix}\Bigr\rangle,
		\]
		where the triple $(a,b,c)$ is taken from the list \eqref{eq:list6} but the $i$th entry for $i \in \{12, 17, 20, 32, 33, 36, 37, 38\}$ is omitted.
	\end{proposition}
	\begin{proof}
		We already computed that all escalations of $E_3^{(6)}$ are either in the list \eqref{eq:list6} or correspond to one of the triples $(1,2+\sqrt2,2+\sqrt2)$ and $(1+\sqrt2,2+\sqrt2,\sqrt2)$. In the latter case, we checked in Example \ref{ex:rankdidntgrow} that both triples correspond to $\qf{1,1,2+\sqrt2}$, which is universal by \cite{CKR} (and already occurred as $E_3^{(1)}$ in Proposition \ref{pr:ternaryesc}).
		
		Let us denote the quaternary lattice corresponding to the $i$th triple, $1 \leq i \leq 38$, as $E_{4,(6)}^{(i)}$. By exchanging the last basis vector $\vv$ for $\vv-a\ee_1$, it can be expressed in the form $\qf{1} \perp N$ where $N$ is a ternary lattice, as given in the statement of this proposition. One can easily check, usually just by multiplying some of the basis vectors by a unit, that there are the following isometries:
		\[
		E_{4,(6)}^{(i)} \simeq E_{4,(6)}^{(j)} \text{ for $(i,j) \in M$,}
		\]
		where $M = \{(6,12), (16,17), (19,20), (22,32), (24,33), (27,36), (28,37), (29,38)\}$. On the other hand, one can check (on a computer or with a lot of patience) that there are no other isometries, so there are indeed $30 = 38 - 8$ distinct quaternary escalations. They are all universal by Proposition \ref{pr:6universality}.
	\end{proof}
	
	We had postponed the proof of universality of the quaternary escalations of $E_{3}^{(6)}$, i.e.\ of Proposition \ref{pr:6universality}, for later; with this exception, we have handled the branch of the escalation tree starting from $E_{3}^{(6)}$. Now we turn our attention to the other non-universal ternary escalator $E_{3}^{(5)} \simeq \qf{1,2+\sqrt2,3}$. Since its truant is also $3(2+\sqrt2)$, its escalations correspond to matrices
	\begin{equation*} 
		\G_{4,(5)} = \begin{pmatrix}
			1 & 0 & 0 & a\\
			0 & 2+\sqrt2 & 0 & b \\
			0 & 0 & 3 & c\\
			a & b & c & 3(2+\sqrt2)\\
		\end{pmatrix}.
	\end{equation*}
	Similarly as before, we compute all triples $(a,b,c)$ up to sign changes for which this matrix is totally positive semidefinite. In this case it will in fact always be totally positive definite, which is due to the fact that there exists no classical ternary overlattice of $E_{3}^{(5)}$. There are $53$ such triples:
	\newcommand{\sepp}{,}
	\begin{alignat}{10}\label{eq:list5}
		&(0,0,0)\sepp &&(0,0,1)\sepp &&(0,0,\ve)\sepp &&(0,0,\!\sqrt2)\sepp &&(0,0,\pi_2)\sepp &&(0,0,2)\sepp &&(0,0,2\ve)\sepp &&(0,\!0,\!1{+}\pi_2)\sepp &&(0,1,0)\sepp &&(0,\ve,0), \nonumber\\
		&(0,\ve,1)\sepp &&(0,\ve,\ve)\sepp &&(0,\ve,\!\sqrt2)\sepp &&(0,\ve,\pi_2)\sepp &&(0,\ve,2)\sepp &&(0,\ve,2\ve)\sepp &&(0,\!\ve,\!1{+}\pi_2)\sepp &&(0,\ve^2,0)\sepp &&(0,\!\pi_2,0)\sepp &&(0,\!\pi_2,1), \nonumber\\
		&(0,\pi_2,\ve)\sepp &&(0,\!\pi_2,\!\!\sqrt2)\sepp &&(0,\pi_2,\pi_2)\sepp &&(0,\!\pi_2,\!1{+}\pi_2)\sepp &&(0,2\ve,0)\sepp &&(0,2\ve,1)\sepp &&(0,2\ve,\ve)\sepp &&(1,0,0)\sepp &&(1,0,1)\sepp &&(1,0,\ve), \nonumber\\
		&(1,\!0,\!\sqrt2)\sepp &&(1,0,\pi_2)\sepp &&(1,0,2\ve)\sepp &&(1,\ve,0)\sepp &&(1,\ve,1)\sepp &&(1,\ve,\ve)\sepp &&(1,\ve,\pi_2)\sepp &&(1,\pi_2,0)\sepp &&(1,\!\pi_2,\ve)\sepp &&(1,\!\pi_2,\!\pi_2), \nonumber\\
		&(\ve,0,0)\sepp &&(\ve,0,1)\sepp &&(\ve,0,\ve)\sepp &&(\ve,0,\!\sqrt2)\sepp &&(\ve,0,\pi_2)\sepp &&(\ve,0,2)\sepp &&(\ve,\ve,0)\sepp &&(\ve,\ve,1)\sepp &&(\ve,\ve,\ve)\sepp &&(\ve,\ve,\!\sqrt2), \nonumber\\
		&(\ve,\pi_2,0)\sepp &&(\ve,\pi_2,1)\sepp &&(\ve,\pi_2,\!\!\sqrt2).
	\end{alignat}
	
	We will denote the lattice corresponding to $i$th triple in this list as $E_{4,(5)}^{(i)}$. Most of them seem to be universal, but because $E_3^{(5)}$ has class number $2$, we are unable to repeat our success from the previous branch and prove the universality for all of them. Only the second and third lattice are clearly not universal and require further escalation.
	
	\begin{proposition} \label{pr:5summary}
		All escalations of $E_3^{(5)}$ are quaternary lattices; up to isometry, there are 38 of them. They are
		\[
		\qf{1} \perp \Bigl\langle
		\begin{smallmatrix}
			2+\sqrt2 & 0 & b \\
			0 & 3 & c\\
			b & c & 3(2+\sqrt2)-a^2
		\end{smallmatrix}\Bigr\rangle,
		\]
		where the triple $(a,b,c)$ is taken from the list \eqref{eq:list5}, but the fifteen entries at positions $18$, $29$, $31$, $33$, $35$, $40$, $43$--$46$ and $49$--$53$ are omitted.
		
		Of these, $E_{4,(5)}^{(2)}$ has truant $3(3-\sqrt2)$ and $E_{4,(5)}^{(3)}$, which is its conjugate, has truant $3(3+\sqrt2)$; they represent all other totally positive elements with norm under $\boundII{}$. The remaining 36 escalations of $E_3^{(5)}$ represent all elements with norm under $\boundII{}$.
	\end{proposition}
	\begin{proof}
		We already know that the $53$ lattices $E_{4,(5)}^{(i)}$ are the only candidates for escalations of $E_3^{(5)}$ and that they can be transformed into the given form by a slight change of basis. We list the isometries which allow us to get rid of $15$ lattices: 
		\[
		E_{4,(5)}^{(i)} \simeq E_{4,(5)}^{(j)} \text{ for $(i,j) \in M$,}
		\]
		where $M = \{(9,18), (6,29), (27,31), (27,33), (15,35), (24,40), (7,43), (8,44),\allowbreak (26, 45),\allowbreak (26,46),\allowbreak (16,49),\allowbreak (17,50),\allowbreak (38,51), (39,52), (24,53)\}$. On the other hand, one can check that there are no other isometries, other than the implied ones $(31,33), (40,53), (45,46)$, which we did not list since they do not help us get rid of some lattices.
		
		We used a program written in Magma to compute all elements up to norm $\boundII{}$ and check their representability.
	\end{proof}
	
	As mentioned, of the $38$ quaternary escalators $E_{4,(5)}^{(i)}$, $36$ seem to be universal. Only for $i=2,3$ we get a pair of conjugate lattices which seem to represent everything except $3(3\mp\sqrt2)$. We will present evidence (and partial proofs) for these statements later; for now let us assume them to be correct. Then every escalation of $E_{4,(5)}^{(2)}$ or $E_{4,(5)}^{(3)}$ is universal. Just for the sake of completeness we computed all of them, so that we can say:
	
	\begin{proposition} \label{pr:quinaryescalators}
		Up to isometry, there are 1656 escalations of $E_{4,(5)}^{(2)}$. They all have rank $5$, and assuming our Conjecture \ref{co:13lattices}\ref*{it:co1} that $E_{4,(5)}^{(2)}$ represents all but $3(3-\sqrt2)$, they are all universal. Analogous statement holds for escalations of the conjugate lattice $E_{4,(5)}^{(3)}$.
	\end{proposition}
	
	Now, assuming that there is no lattice with truant of norm more than $\boundII{}$, we have found all escalators and hence also all critical elements. In the remainder of the section, we gather evidence for our conjectures that all quaternary escalators are universal except for the two which represent everything except for their truant. But observe that we have already proved Theorem \ref{th:sqrt2main}, parts \ref*{it:first}--\ref*{it:third}. The following proposition makes a big step towards proving part \ref*{it:last}.
	
	\begin{proposition} \label{pr:6universality}
		All the 30 quaternary escalations of $E_3^{(6)}$ are universal.
	\end{proposition}
	
	We prove it immediately after proving Lemma \ref{le:diagsublat}, which contains the bulk of the proof. We also note that for several of the lattices one can use an easier proof of universality than the one we give below -- namely, finding a ternary sublattice isometric to one of the four universal lattices $E_3^{(j)}$, $1\leq j \leq 4$. To be specific, we have $E_3^{(j)}$ as a sublattice of $E_{4,(6)}^{(i)}$ if $(i,j)$ if one of $(6,2)$, $(19,2)$, $(25,4)$, $(26,4)$, $(28,1)$, $(29,1)$, $(34,3)$, $(35,3)$, which gives a direct proof of the universality of $E_{4,(6)}^{(i)}$ for these eight values of $i$, and also $(12,2)$, $(20,2)$, $(37,1)$, $(38,1)$ which does not help us, as these values of $i$ were already eliminated thanks to isometry of lattices, see Proposition \ref{pr:6summary}.
	
	For another pair of isometric lattices, $E_{4,(6)}^{(27)} \simeq E_{4,(6)}^{(36)}$, one can compute that the class number is $1$, so in this case the universality can be proven directly using the local--global principle. But in this way we only handled $8+1$ of the $30$ escalators, so we might as well use the proof based on the lemma below for all $30$ lattices. The strategy is to find a diagonal sublattice which represents all elements with large enough norm.
	
	\begin{lemma} \label{le:diagsublat}
		Recall that $E_3^{(6)}\simeq \qf{1,2+\sqrt2,2}$.
		\begin{enumerate}
			\item Any $\alpha \in \Z[\sqrt2]$ is represented by $E_3^{(6)}$ if and only if $\alpha \succeq 0$ and $\alpha \not \equiv 6 \pm 3\sqrt2 \pmod8$. \label{it:diagsublat1}
			\item Let $D = E_3^{(6)} \perp \qf{\gamma}$, $\gamma = m + n(2+\sqrt2)$, where $\gamma \not \equiv 0 \pmod8$ and $\gamma \not\equiv \pm 2\sqrt2 \pmod8$. Let $\alpha$ be a totally positive element of $\Z[\sqrt2]$ not represented by $D$. Then $\NN(\alpha) < 17m^2 + 52mn + 34n^2$. \label{it:diagsublat2} 
		\end{enumerate}
	\end{lemma}
	\begin{proof}
		For the first part, we compute that the class number of $E_3^{(6)}$ is $1$, since it has $8$ automorphisms and mass $1/8$. Therefore it satisfies the local--global principle. The real embeddings yield $\alpha \succeq 0$. At the non-dyadic places the lattice is universal, since it is a unimodular ternary lattice. It remains to compute the conditions at the completion given by the dyadic prime $\sqrt2$. After some computation, one arrives precisely at the condition $\alpha \not \equiv 6 \pm 3\sqrt2 \pmod8$.
		
		The second part relies heavily on the first. First of all, any element not represented by $D$ is not represented by $E_3^{(6)}$ either, so we only have to consider $\alpha \equiv 6 \pm 3\sqrt2 \pmod8$. Now let us introduce some notation: In accordance with \cite{HK}, we order the indecomposable elements of $\Z[\sqrt2]$ by their absolute value to form the bi-infinite sequence $\ldots, \beta_{-1}=2-\sqrt2, \beta_0=1, \beta_1=2+\sqrt2, \beta_2 = 3+2\sqrt2, \ldots$ We have $\beta_i = \ve^2 \beta_{i-2}$ where $\ve = 1+\sqrt2$ is the fundamental unit. From \cite{HK} we know that every totally positive element can be written as a linear combination of two consecutive indecomposables with nonnegative coefficients; after multiplication by a suitable power of $\ve^2$, we can assume $\alpha = k\beta_0 + l\beta_1$ or $k\beta_0 + l\beta_{-1}$ for $k,l \geq 0$. (And we still have $\alpha \equiv 6 \pm 3\sqrt2 \pmod8$, as multiplication by a square of a unit does not influence representability by a given lattice.) Now we make the following crucial claim: \emph{In both cases, if either $k \geq m + 4n$ or $l \geq 2m + n$, then at least one of $\alpha - \gamma$ and $\alpha - \ve^{-2}\gamma$ is totally nonnegative.}
		
		Assume first that this claim is correct. Then we distinguish two situations. If $k \geq m + 4n$ or $l \geq 2m + n$, then we can use the claim. Since $\alpha \equiv 6\pm 3\sqrt2$ but $\gamma$ and thus also $\ve^{-2}\gamma$ is neither $0$ nor $\pm 2\sqrt2 \pmod8$, the totally nonnegative element $\alpha -\gamma$ or $\alpha - \ve^{-2}\gamma$ is not $6\pm 3\sqrt2 \pmod8$, hence it satisfies the conditions of the first part of this lemma, hence it is represented by $E_3^{(6)}$. Thus $\alpha$ is represented by $D$. It remains to look at the situation when $k < m + 4n$ and $l < 2m + n$; but then $\NN(k\beta_0 + l\beta_1) = \NN(k\beta_0 + l\beta_{-1}) = (k+2l)^2-2l^2 = k^2 + 4kl + 2l^2 < 17m^2 + 52mn + 34n^2$. This proves the lemma.
		
		It remains to prove the claim. We start with the case $\alpha = k\beta_0 + l\beta_1$. Thus $\alpha - \gamma = (k-m)\beta_0 + (l-n)\beta_1$. We distinguish three situations. If $k \geq m$ and $l \geq n$, then $\alpha-\gamma$ is clearly totally nonnegative. Assume now $k < m$. (This in particular means that from the assumptions of the claim, the former is not satisfied, so $l\geq 2m+n$. Thus, trivially, $2k+l\geq 2m+n$.) Since $\beta_0+\beta_2 = 2\beta_1$, we can rewrite $\alpha - \gamma = (l-n-2m+2k)\beta_1 + (m-k)\beta_2$. Since $2k+l \geq 2m+n$, this is totally positive. Finally, assume $l < n$. (In this case the latter assumption of the claim cannot hold, so we get $k\geq m+4n$ and thus $k+4l\geq m+4n$.) Now we use $\beta_{-1}+\beta_1 = 4\beta_0$ to rewrite $\alpha - \gamma = (n-l)\beta_{-1} + (k-m-4n+4l)\beta_0$. And this is totally positive since $k+4l \geq m+4n$. 
		
		Similarly, we handle the case $\alpha = k\beta_0 + l\beta_{-1}$. Then $\alpha - \gamma = l\beta_{-1} + (k-m)\beta_0 - n\beta_1 = (l+n)\beta_{-1} + (k-m-4n)\beta_0$. Thus, if $k \geq m + 4n$, then $\alpha - \gamma \succeq 0$. Also, $\alpha - \ve^{-2}\gamma = -m\beta_{-2} + (l-n)\beta_{-1} + k\beta_0 = (l-n-2m)\beta_{-1} + (k+m)\beta_0$. Thus, if $l \geq n+2m$, then $\alpha - \ve^{-2}\gamma \succeq 0$. This proves the claim.
	\end{proof}
	
	We use the norm just for the simplicity of presentation; from the proof it is clear that one could be more explicit in the statement of the Lemma, namely: \emph{Under the assumptions of part \ref*{it:diagsublat2}, up to multiplication by squares of units, we have $\alpha=k+l(2+\sqrt2)$ or $\alpha=k+l(2-\sqrt2)$ where $0\leq k < m+4n$ and $0\leq l < 2m+n$.} This yields a very explicit finite list of elements that are potentially not represented by the form $D$.
	
	Now that we understand diagonal lattices of the form $\qf{1,2+\sqrt2,2,\gamma}$, we are ready to prove universality of all escalations of $E_3^{(6)}$.
	
	\begin{proof}[Proof of Proposition \ref{pr:6universality}]
		For each of the lattices $E_{4,(6)}^{(i)}$ we find a sublattice of the form $E_3^{(6)} \perp \qf{\gamma}$. From the construction, we have $E_3^{(6)}$ as a sublattice, so we can just compute the orthogonal complement in $E_{4,(6)}^{(i)}$, i.e.\ the set of all vectors which have zero scalar product with every vector in $E_{4,(6)}^{(i)}$. This is a straightforward calculation. Below we list the resulting $\gamma$ expressed in the form suitable for application of Lemma \ref{le:diagsublat}; we keep the notation $\beta_0=1$, $\beta_1=2+\sqrt2$, $\beta_{-1}=2-\sqrt2$. We give $\gamma_i$ for all $38$ values of $i$; we could omit eight of them thanks to the isometries listed in Proposition \ref{pr:6summary}, but then we would need to reorder the lattices.
		\newcommand{\seppp}{,\,}
		\begin{alignat*}{10} 
			&3\beta_1\seppp &&2\beta_0{+}8\beta_{-1}\seppp &&2\beta_0{+}8\beta_1\seppp &&2\beta_0{+}2\beta_{-1}\seppp &&2\beta_0{+}2\beta_1\seppp &&\beta_1\seppp &&5\beta_1\seppp &&2\beta_0{+}6\beta_{-1}\seppp &&2\beta_0{+}6\beta_1\seppp &&2\beta_0{+}\beta_{-1}, \\
			&2\beta_0{+}\beta_1\seppp &&\beta_1\seppp &&2\beta_1\seppp &&2\beta_0{+}4\beta_{-1}\seppp &&2\beta_0{+}4\beta_1\seppp &&2\beta_0\seppp &&2\beta_0\seppp &&\beta_1\seppp &&2\beta_0\seppp &&2\beta_0, \\
			&\beta_0{+}\beta_{-1}\seppp &&6\beta_0\seppp &&2\beta_0{+}4\beta_{-1}\seppp &&2\beta_1\seppp &&2\beta_0{+}\beta_{-1}\seppp &&2\beta_0{+}2\beta_{-1}\seppp &&\beta_1\seppp &&\beta_0\seppp &&2\beta_0\seppp &&\beta_0{+}\beta_1, \\
			&2\beta_0{+}4\beta_1\seppp &&6\beta_0\seppp &&2\beta_1\seppp &&2\beta_0{+}\beta_1\seppp &&2\beta_0{+}2\beta_1\seppp &&\beta_1\seppp &&\beta_0\seppp &&2\beta_0
		\end{alignat*}
		None of these values give the forbidden residues $0$ or $\pm2\sqrt2$ modulo $8$. Thus, for those lattices where $\gamma_i = m\beta_0 + n\beta_1$, a direct application of Lemma \ref{le:diagsublat} shows that all elements $\alpha\succeq 0$ with $\NN(\alpha) \geq 17m^2+52mn+34n^2$ are represented. For those where $\gamma_i = m\beta_0 + n\beta_{-1}$, we prove the same: For every $\alpha\succeq 0$ with $\NN(\alpha) \geq 17m^2+52mn+34n^2$, the Lemma gives representability of its conjugate $\overline{\alpha}$ by $\overline{D} \simeq E_3^{(6)} \perp \qf{\overline{\gamma_i}}$ thanks to $\overline{\gamma_i} = m\beta_0 + n\beta_1$ and to the fact that $E_3^{(6)}$ is self-conjugate. So $\alpha$ is indeed represented by $D$.
		
		So for all lattices there remain only finitely many elements for which the representability must be tested, namely those with $\NN(\alpha) < 17m^2 + 52mn + 34n^2$. These are not necessarily represented by $D$ (although in fact we computed that there are only fourteen elements which fail to be represented by at least one of the 38 forms $D$, namely $6+3\sqrt2$, $14\pm5\sqrt2$, $14\pm3\sqrt2$, $22+11\sqrt2$, $22\pm5\sqrt2$, $22\pm3\sqrt2$, $30\pm5\sqrt2$, $30\pm3\sqrt2$), but one can directly check their representability by the original lattice $E_{4,(6)}^{(i)}$. By far the largest value of $17m^2 + 52mn + 34n^2$ is $3076$, obtained for $i=2$ and $i=3$ where $m=2$ and $n=8$. So we concluded the proof by checking that every $E_{4,(6)}^{(i)}$ represents all elements with norm below $3076$.
	\end{proof}
	
	Now we switch our attention to escalations of $E_3^{(5)}$. We listed them in Proposition \ref{pr:6summary}. As mentioned, we are unable to prove universality of many of them, since $E_3^{(5)}$ does not satisfy the local--global principle.
	
	\begin{remark} \label{re:l-gfails}
		If the elements represented by $E_3^{(5)}$ were known (see the next paragraph of this remark), it should be possible to prove Conjecture \ref{co:13lattices} and thus also Conjecture \ref{co:mainsqrt2} by using the same strategy as for $E_3^{(6)}$ in the proof of Proposition \ref{pr:6universality}. But even in that situation (and recall that proving which elements are represented by a ternary form is notoriously hard), the technical details would be much more tedious for two reasons: 1) The local obstructions for representation by $E_3^{(5)}$ come from two primes, $\sqrt2$ and $3$, and thus they are more complicated than the conditions for $E_3^{(6)}$. 2) The local obstructions are not enough; there are at least twelve elements that are represented locally but not globally. This would require extra care if we wanted to prove the analogue of Lemma \ref{le:diagsublat}\ref*{it:diagsublat2}.
		
		We computed the following: \emph{Let $\alpha$ be locally represented by $E_3^{(5)}$ and $\NN(\alpha) \leq 200\,000$. Then $\alpha\to E_3^{(5)}$ if and only if $\alpha$ is not among the nine exceptions $7\pm3\sqrt2$, $3(3\pm\sqrt2)$, $9(7\pm3\sqrt2)$, $51$ and $27(3\pm\sqrt2)$.}
		
		The largest norm in this list of exceptions is $5103$. This might inspire some confidence that this list is complete. However, it turns out that there are at least three other exceptional elements, namely $9^2(7\pm3\sqrt2)$ and $9\cdot 51$, with norms $203\,391$ and $210\,681$, respectively. 
		
		Of course in hindsight one sees that $9$ times an exception may still very well be an exception, but just looking at the norms of the exceptions it is easy to overlook this. We present this as a warning against blindly trusting to numerical evidence.
	\end{remark}
	
	For all the reasons listed in the remark above, we chose the following approach: Instead of making one conjecture about the ternary escalator $E_3^{(5)}$, we will explicitly assume universality of some of the quaternary escalators. To make this conjecture immediately understandable to anyone who has not read the whole section, we list the lattices explicitly. We also use this opportunity and present them in a nicer form than the one given in Proposition \ref{pr:5summary} by partially diagonalising some of the lattices -- note that there are only two $3 \times 3$ matrices in the formulation.
	
	\begin{conjecture} \label{co:13lattices}
		Consider the quaternary lattices
		\begin{align*}
			&E_{4,(5)}^{(1)}\simeq\qf{1,2+\sqrt2,3,3(2+\sqrt2)},& \quad & E_{4,(5)}^{(2)}\simeq\qf{1,2+\sqrt2} \perp\bigl\langle\begin{smallmatrix}
				3 & 1 \\
				1 & 3(2+\sqrt2)  \\
			\end{smallmatrix}\bigr\rangle, \\ 
			&E_{4,(5)}^{(4)}\simeq\qf{1,2+\sqrt2} \perp\bigl\langle\begin{smallmatrix}
				3 & \sqrt2 \\
				\sqrt2 & 3(2+\sqrt2)  \\
			\end{smallmatrix}\bigr\rangle,
			&\quad & E_{4,(5)}^{(7)}\simeq\qf{1,2+\sqrt2} \perp\bigl\langle\begin{smallmatrix}
				3 & 1+\sqrt2 \\
				1+\sqrt2 & 3+\sqrt2  \\
			\end{smallmatrix}\bigr\rangle,\\
			&E_{4,(5)}^{(8)}\simeq\qf{1,2+\sqrt2} \perp\bigl\langle\begin{smallmatrix}
				3 & \sqrt2 \\
				\sqrt2 & 3+\sqrt2  \\
			\end{smallmatrix}\bigr\rangle,
			&\quad & E_{4,(5)}^{(10)}\simeq\qf{1,3} \perp\bigl\langle\begin{smallmatrix}
				2+\sqrt2 & 1 \\
				1 & 3(2-\sqrt2)  \\
			\end{smallmatrix}\bigr\rangle,\\
			&E_{4,(5)}^{(11)}\simeq\qf{1} \perp\Bigl\langle\begin{smallmatrix}
				2+\sqrt2 & 0 & 1+\sqrt2 \\
				0 & 3 & 1 \\
				1+\sqrt2 & 1 & 3(2+\sqrt2)  \\
			\end{smallmatrix}\Bigr\rangle, & \quad &
			E_{4,(5)}^{(13)}\simeq\qf{1} \perp\Bigl\langle\begin{smallmatrix}
				2+\sqrt2 & 0 & 1+\sqrt2 \\
				0 & 3 & \sqrt2 \\
				1+\sqrt2 & \sqrt2 & 3(2+\sqrt2)  \\
			\end{smallmatrix}\Bigr\rangle,\\
			&E_{4,(5)}^{(19)}\simeq\qf{1,2+\sqrt2,3,2(2+\sqrt2)},& \quad &
			E_{4,(5)}^{(20)}\simeq\qf{1,2+\sqrt2} \perp\bigl\langle\begin{smallmatrix}
				3 & 1 \\
				1 & 2(2+\sqrt2)  \\
			\end{smallmatrix}\bigr\rangle,\\
			&E_{4,(5)}^{(22)}\simeq\qf{1,2+\sqrt2} \perp\bigl\langle\begin{smallmatrix}
				3 & \sqrt2 \\
				\sqrt2 & 2(2+\sqrt2)  \\
			\end{smallmatrix}\bigr\rangle,& \quad &
			E_{4,(5)}^{(41)}\simeq\qf{1,2+\sqrt2,3,3+\sqrt2},\\
			&E_{4,(5)}^{(42)}\simeq\qf{1,2+\sqrt2} \perp\bigl\langle\begin{smallmatrix}
				3 & 1 \\
				1 & 3+\sqrt2  \\
			\end{smallmatrix}\bigr\rangle.    
		\end{align*}
		Then:
		\begin{enumerate}
			\item $E_{4,(5)}^{(2)}$ represents all $\alpha \in \Z[\sqrt2]^+$ except for $3(3-\sqrt2) \ve^{2k}$, $k \in \Z$. \label{it:co1}
			\item $E_{4,(5)}^{(i)}$ is universal for $i = 1, 4, 7, 8, 10, 11, 13, 19, 20, 22, 41, 42$.
		\end{enumerate}
	\end{conjecture} 
	
	Note that thanks to the following lemma, it is enough to test only these $13$ lattices up to our bound $\boundII{}$ in order to prove Proposition \ref{pr:5summary}, instead of all $38=53-15$ escalations of $E_3^{(5)}$. It is also worth mentioning that we could have ignored the escalators $E_{4,(5)}^{(24)}$ and $E_{4,(5)}^{(39)}$, as they are isometric to $E_{4,(6)}^{(29)}$ and $E_{4,(6)}^{(28)}$, respectively. Therefore, there are only $30+38-2 = 66$ quaternary escalators up to isometry ($64$ of them probably universal).
	
	\begin{lemma} \label{le:conjecturesuffices}
		Assuming Conjecture \ref{co:13lattices}, all escalations of $E_3^{(5)}$ are universal except for $E_{4,(5)}^{(2)}$ and $E_{4,(5)}^{(3)}$, both of which fail to represent exactly one element of $\Z[\sqrt2]^+/\U_{\qq{2}}^2$.
	\end{lemma}
	\begin{proof}
		We need to handle all escalators $E_{4,(5)}^{(i)}$ for $1 \leq i \leq 53$. Remember that $i=18$, $29$, $31$, $33$, $35$, $40$, $43$--$46$ and $49$--$53$ can be ignored by Proposition \ref{pr:5summary}, since they occur in the list repeatedly. For $i=2$, the conjecture directly yields the statement, and for $i=3$ also, since the lattices $E_{4,(5)}^{(2)}$ and $E_{4,(5)}^{(3)}$ are conjugate.
		
		Many lattices can be proved to be universal unconditionally, since they contain one of the universal ternary lattices $E_3^{(j)}$, $1\leq j \leq 4$. In particular, the universal lattice with $j=1$ occurs as a sublattice for $i = 24$, $38$, $39$; the sublattice with $j=2$ can be found in $i=9$; the sublattice $j=3$ solves $i=16$, $17$, $26$, $47$, $48$; and $j=4$ solves $i=15$, $27$, $34$, $36$, $37$.
		
		So from the $38-2$ isometry classes of escalators for which we need to prove universality, we successfully handled $14$. For further reduction, we observe that many of the remaining $22$ escalators form conjugate pairs, so it is enough to explicitly assume universality for only one of them. The lattices with $i=1$, $10$, $19$ and $25$ are self-conjugate, but for the others, the conjugate pairs are $(4,5)$, $(7,6)$, $(8,32)$, $(11,12)$, $(13,14)$, $(20,21)$, $(22,23)$, $(41,28)$ and $(42,30)$; we ordered them in such a way that the first number is listed in Conjecture \ref{co:13lattices}, so the universality of the other follows.
		
		From the self-conjugate lattices with $i=1$, $10$, $19$ and $25$, the first three are also explicitly assumed universal in the conjecture. Universality of the last one follows from the fact that $E_{4,(5)}^{(25)} \simeq \qf{1,2+\sqrt2,3,2+\sqrt2}$ contains $E_{4,(5)}^{(19)} \simeq \qf{1,2+\sqrt2,3,2(2+\sqrt2)}$, which is assumed to be universal.
	\end{proof}

	\section{Universality criterion set for \texorpdfstring{$\qq{3}$}{Q(√3)}} \label{se:sqrt3}
	
	As before, \emph{throughout the section, all lattices are classical}. This section studies the classical criterion set over $\qq{3}$. The route by which we arrive at our conclusions is somewhat different from the previous section where we applied the escalation procedure very straightforwardly; but before discussing the method, let us present our conjectural results.
	
	\begin{conjecture}\label{co:sqrt3}
		Denote $\ve = 2+\sqrt3$. We have
		\[
		\Ccl_{\qq{3}} = \Cdiag_{\qq{3}} = \bigl\{1, \ve, 4\pm\sqrt3, \ve(4\pm\sqrt3), 7\pm2\sqrt3, \ve(7\pm2\sqrt3)\bigr\}.
		\]
	\end{conjecture}
	
	As before, we state how much we can actually prove.
	
	\begin{theorem} \label{th:sqrt3}
		Let $\CC = \bigl\{1, \ve, 4\pm\sqrt3, \ve(4\pm\sqrt3), 7\pm2\sqrt3, \ve(7\pm2\sqrt3)\bigr\}$ where $\ve=2+\sqrt3$. Then:
		\begin{enumerate}
			\item $\CC \subset \Cdiag_{\qq{3}}\subset\Ccl_{\qq{3}}$.
			\item $\Ccl_{\qq{3}}$ contains no other element with norm below $\boundIII{}$.
			\item $\CC$ is the $S$-criterion set for $S = \{\alpha \in \Z[\sqrt3]^{+} \mid \NN(\alpha) \leq \boundIII{}\}$.
			\item Conjecture \ref{co:sqrt3} holds if Conjectures \ref{co:sqrt3ternary} and \ref{co:5lattices} do, i.e.\ if we assume the knowledge of all elements not represented by one ternary and five quaternary lattices.
		\end{enumerate}
	\end{theorem}
	\begin{proof}
		Since universality criterion sets are closed under multiplication by totally positive units and under conjugation \cite[Thm.~1.2]{KKR}, it is enough to find diagonal forms with truant $1$, with truant $4+\sqrt3$ and with truant $7+2\sqrt3$. These are $\{\vectorstyle{0}\}$, $\qf{1,\ve}$ and $\qf{1,\ve,3+\sqrt3,\varepsilon(4+\sqrt3),4-\sqrt3}$. This proves the first claim.
		
		It remains to prove that assuming the two conjectures, or considering only elements of norm under $\boundIII{}$, there are no other elements of $\Ccl_{\qq{3}}$. Both facts follow from Proposition \ref{pr:sqrt3main}, which contains even finer information about the criterion set.
	\end{proof}
	
	Note that the 10-element set from Conjecture \ref{co:sqrt3} and Theorem \ref{th:sqrt3} consists exactly of the totally positive elements with norms $1$, $13$ and $37$. Now we state a result which strongly indicates that the six elements with norms $1$ and $13$ in fact form $\Ccl_S$ for $S=\{\alpha\in\Z[\sqrt3]^+ \mid \NN(\alpha)\neq 37\}$. For further discussion of similar phenomena, see Section \ref{se:amax}.
	
	\begin{proposition} \label{pr:sqrt3main}
		Let $L$ be a lattice representing all of $1,\ve,4\pm\sqrt3,\ve(4\pm\sqrt3)$. Then, up to norm $\boundIII{}$, $L$ represents all of $\Z[\sqrt3]^+$ except possibly for $7\pm 2\sqrt3$ and $\ve(7\pm 2\sqrt3)$.
		
		Moreover, assuming Conjectures \ref{co:sqrt3ternary} and \ref{co:5lattices}, we can remove the assumption about norm:  $L$ represents all of $\Z[\sqrt3]^+$ except possibly for $7\pm 2\sqrt3$ and $\ve(7\pm 2\sqrt3)$.
	\end{proposition}
	
	Clearly, Proposition \ref{pr:sqrt3main} implies Theorem \ref{th:sqrt3}. The rest of the section is spent by proving the Proposition.
	
	\medskip
	
	Performing the whole escalation procedure without any adjustments (as we did in the previous section) would run into serious problems -- there are simply too many escalators, and we would have to use many statements like Proposition \ref{pr:quinaryescalators} but with much more complicated wording: For a given lattice $\dots$, there are so-and-so-many escalators, and they all have such-and-such property in common. It would be tedious and hard to follow, and moreover, our conditional result would probably have to rely on a conjecture about several hundreds or even thousands of escalators. Instead, we found a way how to compute $\Ccl_{\qq{3}}$ almost by hand by cleverly exploiting the symmetries -- not only conjugation but also multiplication by $\ve$. 
	
	We start by performing the first three steps of the escalation. The zero lattice $\{\vectorstyle{0}\}$ has two truants, namely $1$ and $\ve=2+\sqrt3$. Choosing an ordering with $1 <_\rho \ve$, the only escalation is $\qf{1}$ with unique truant $\ve$, and then the only escalation is $\qf{1,\ve}$. (Indeed, there is no nonzero $a^2 \preceq \ve$, since $\ve$ is indecomposable and not a square.)
	
	The binary escalator $L_2\simeq\qf{1,\ve}$ already represents all elements with norm less than $13$, but it fails to represent all the elements with norm exactly $13$, which we shall write as $2+\ve$, $2+\overline{\ve}$, $2\ve+1$ and $2\overline{\ve}+1$. This is connected with the fact that $L_2$ is very symmetric -- it is isometric to its $\ve$-multiple $(L_2,\ve Q)$ as well as to its conjugate $(L_2,\overline{Q})$.
	
	Thus, $L_2$ has four distinct truants. Let us make one more escalation step, assuming that $2+\ve=4+\sqrt3$ is the smallest of them with respect to $\rho$. Let $\ee_1$ and $\ee_2$ be the standard basis of $L_2$, and let $\vv$ be a vector such that $Q(\vv)=2+\ve$. The corresponding Gram matrix is
	\[
	\begin{pmatrix}
		1 & 0 & a  \\
		0 & \ve & b \\
		a & b & 2+\ve
	\end{pmatrix} \text{ for some $a,b \in \Z[\sqrt3]$}.
	\]
	However, we can replace $\vv$ by the vector $\ww=\vv-a\ee_1-\overline{\ve}b\ee_2$, and we see that the lattice generated by $\ee_1,\ee_2,\vv$ is the same as the lattice generated by $\ee_1,\ee_2,\ww$, which is isometric to $\qf{1,\ve,2+\ve-a^2-\overline{\ve} b^2}$. Since $Q(\ww) \succeq 0$, we easily compute that there are the following possibilities for $Q(\ww)$, and each of them indeed corresponds to some values of $a,b$: $1, 2, \ve, 1+\ve, 2+\ve$.
	
	Let us summarise what we obtained now:
	
	\begin{lemma} \label{le:ternaryescalators}
		There are five possible lattices generated by vectors representing $1$, $\ve$ and $2+\ve$. Of them, two are universal, namely $\qf{1,\ve,1}$ and $\qf{1,\ve,\ve}$. The remaining three are $\qf{1,\ve,1+\ve}$, $\qf{1,\ve,2}$ and $\qf{1,\ve,2+\ve}$.
	\end{lemma}
	\begin{proof}
		In light of the discussion above, it only remains to prove universality of the two escalators. They are precisely the two universal ternary lattices over $\qq{3}$ given in \cite{CKR}. (Also, they both have class number $1$, so it would be enough to check their universality locally.)
	\end{proof}
	
	Now we are left with three non-universal escalators, and we could start going through all their escalations step by step; however, we already hinted that this would be tedious at best, and possibly unfeasible. Instead, based on very strong computational evidence, we conjecture that one of these lattices is already almost universal:
	
	\begin{conjecture} \label{co:sqrt3ternary}
		The quadratic form $\qf{1,\ve,1+\ve}$ represents all of $\Z[\sqrt3]^+$ except for $2+\overline{\ve}$, $1+2\overline{\ve}$, $2+3\ve$ and $3+2\ve$ (with norms $13$, $13$, $37$ and $37$).    
	\end{conjecture}
	
	Since the computations in Magma language for this ternary lattice ran quite fast, we verified the above conjecture up to norm one million:
	
	\begin{lemma} \label{le:boundHuge}
		Conjecture \ref{co:sqrt3ternary} is true for all elements with norm at most $\boundHuge{}$.
	\end{lemma}
	
	Remember that our aim is to prove Proposition \ref{pr:sqrt3main}; now it is clear that Conjecture \ref{co:sqrt3ternary} and Lemma \ref{le:boundHuge} prove it for all lattices of the form $\qf{1,\ve} \perp L$ where $L$ represents $1+\ve$. (Or $1+\overline{\ve}$, thanks to symmetry.)
	
	As a second preparatory step, let us look at the escalator $\qf{1,\ve,2}$. Fortunately, it has class number $1$, so we know exactly which elements are represented.
	
	\begin{proposition} \label{pr:regularternary}
		Let $\alpha\in\Z[\sqrt3]^+$. Then $\alpha \to \qf{1,\ve,2}$ if and only if $\alpha \not\equiv \pm(1+2\sqrt3) \pmod4$.
	\end{proposition}
	\begin{proof}
		Since the lattice has class number $1$, it satisfies the local--global principle. The determinant is $2\ve$, so at all non-dyadic places, it is ternary unimodular and thus universal. Regarding the dyadic prime $1+\sqrt3$, a computation is needed, but in the end one arrives precisely at one forbidden class, namely $\pm(1+2\sqrt3)$ modulo $(1+\sqrt3)^4$, which is the same as modulo $4$.
	\end{proof}
	
	Instead of computing the escalations of this lattice, we switch to the one ternary escalator which we did not study at all, namely $\qf{1,\ve,2+\ve}$. It still has three truants of norm $13$, namely $2+\overline{\ve}$, $2\ve+1$ and $2\overline{\ve}+1$. Let us pick the truant $2+\overline{\ve}$ and continue with our computations. We shall not compute all the escalations, only those of a quite specific form -- the reason why we can ignore the others will be explained later.
	
	\begin{lemma} \label{le:2+eps2+bareps}
		Let $L$ be a lattice of the form $\qf{1,\ve} \perp J$ where $J\simeq\bigl\langle\begin{smallmatrix} 2+\ve & a \\ a & 2+\overline{\ve} \end{smallmatrix}\bigr\rangle$. This lattice is totally positive definite if and only if $a \in \{0,\pm1,\pm2,\pm3,\pm\sqrt{3},\pm(1+\sqrt3),\pm(1-\sqrt3),\pm2\sqrt3\}$.
		
		The lattices in every $\pm$ pair are isometric. 
		Moreover, if $a=2\sqrt3$, then $\ve\to J$ (in particular, $L$ is universal); and if $a=3$, then $2\to J$. The cases $a=1+\sqrt3$ and $a=1-\sqrt3$ correspond to a pair of conjugate lattices.
	\end{lemma}
	\begin{proof}
		First observe that a switch from $a$ to $-a$ corresponds to multiplying the last basis vector by $-1$; this of course does not change the lattice, implying isometry. Thus we can assume that $a=x+y\sqrt3$ with $x\geq0$.
		
		The $2\times 2$ matrix is totally positive semidefinite if and only if $a^2\preceq (2+\ve)(2+\overline{\ve})=13$. This is equivalent to both embeddings of $a$ being at most $\sqrt{13}$ in absolute value, or equivalently $x+|y|\sqrt3 \leq \sqrt{13}$. One easily lists all pairs $x,y$ satisfying this, which proves the first part of the lemma.
		
		The representation of $\ve$ is readily checked by plugging $X=-1$ and $Y=1+\sqrt3$ into $(2+\ve)X^2+2\cdot2\sqrt3 XY+(2+\overline{\ve})Y^2$; universality follows from universality of $\qf{1,\ve,\ve}$, see Lemma \ref{le:ternaryescalators}. Similarly, the representation of $2$ can be checked by plugging $X=1$ and $Y=-1$ into $(2+\ve)X^2+2\cdot 3XY+(2+\overline{\ve})Y^2$. Finally, the conjugacy of the two lattices follows from $\qf{1,\ve}\simeq\qf{1,\overline{\ve}}$ together with conjugacy of the corresponding $J$'s.
	\end{proof}
	
	Thus, after our reduction there remain six lattices to be studied, two of which form a conjugate pair. It shall turn out that it is enough to conjecture the behaviour of these $6-1$ lattices to conclude the whole proof.
	
	\begin{conjecture} \label{co:5lattices}
		Let $L$ be one of the five lattices of the form $\qf{1,\ve} \perp \bigl\langle\begin{smallmatrix} 2+\ve & a \\ a & 2+\overline{\ve} \end{smallmatrix}\bigr\rangle$ where $a \in \{0,1,2,\sqrt3,1+\sqrt3\}$. Then $L$ represents all of $\Z[\sqrt3]^+$ except possibly for the nine exceptions:
		\[
		5\pm2\sqrt3, 8\pm2\sqrt3, 15\pm7\sqrt3, 15\pm6\sqrt3, 13.
		\]
	\end{conjecture}
	
	Of course it would be very simple to list the precise subset of the nine exceptions for each of the five lattices, but it is unnecessary for our purposes. Let us just note that none of the five lattices represents the first two elements $1+2\ve$ and $1+2\overline{\ve}$, so these are its truants. Also, the listed exceptional elements have norm at most $169$. 
	
	And again we have a computer-based lemma which verifies the conjecture up to a certain bound.
	
	\begin{lemma} \label{le:boundIII}
		Conjecture \ref{co:5lattices} is true for all elements with norm at most $\boundIII{}$.
	\end{lemma}
	
	Now we are ready to prove the main result of this section. The technical details might be confusing at first, so let us sketch the main idea: Let $L$ be a lattice representing all the six elements of norms $1$ and $13$. In the simple cases, it contains one of the lattices $\qf{1,\ve,1}$, $\qf{1,\ve,\ve}$, $\qf{1,\ve,1+\ve}$ or $\qf{1,\ve,1+\overline{\ve}}$, which are universal or covered by the strong Conjecture \ref{co:sqrt3ternary}. If $L$ contains none of them, then the surprising part comes: We prove that every $\alpha\in\Z[\sqrt3]^+$ is represented either by the lattice $M$ generated by vectors that represent $1,\ve,2+\ve,2+\overline{\ve}$, or by the lattice $N$ generated by vectors that represent $1,\ve,2\ve+1,2\overline{\ve}+1$. This last trick immensely decreases the number of possibilities to be considered, since it makes all the values $B_Q(\vv,\ww)$ where $\vv\in M \setminus N$ and $\ww\in N \setminus M$ immaterial.
	
	We separate a part of this into a technical lemma:
	
	\begin{lemma} \label{le:technical}
		Let $M_1,M_2$ be two lattices, both satisfying the following conditions: $M_i$ represents $1$, $\ve$, $2+\ve$ and $2+\overline{\ve}$; further, $M_i$ contains neither $\qf{1,\ve,1+\ve}$ nor $\qf{1,\ve,1+\overline{\ve}}$. Then the following holds:
		
		Let $\alpha\in\Z[\sqrt3]^+$. Assume either Conjecture \ref{co:5lattices} or $\NN(\alpha)\leq \boundIII{}$. Then $\alpha\to M_1$ or $\ve\alpha\to M_2$.
	\end{lemma}
	\begin{proof}
		For most of the proof, let us write $M_i$ for either of the two lattices; we have exactly the same knowledge about both of them.
		
		First note that if $M_1$ or $M_2$ is universal, then the lemma trivially holds; thus, in the rest of the proof, we can assume that $M_i$ contains neither $\qf{1,\ve,1}$ nor $\qf{1,\ve,\ve}$.
		
		Now, our next aim is to prove: \emph{Either $M_i$ contains $\qf{1,\ve,2}$, or $\alpha\not\to M_i$ implies $\alpha\in R= \{5\pm2\sqrt3, 8\pm2\sqrt3, 15\pm7\sqrt3, 15\pm6\sqrt3, 13\}$.}
		
		This is obvious if $M_i$ contains $\qf{1,\ve,2}$, so we can assume that it does not.
		
		Consider now the sublattice of $M_i$ generated by vectors representing $1$, $\ve$, $2+\ve$. We know all the five possibilities by Lemma \ref{le:ternaryescalators}; however, under our current assumptions the only valid possibility is $\qf{1,\ve,2+\ve}$. Now fix the two vectors representing $1$ and $\ve$, and consider the sublattice generated by them together with the vector representing $2+\overline{\ve}$. By applying a \enquote{conjugated version of Lemma \ref{le:ternaryescalators}}, we again know that there are only five possibilities, and again, four of them are forbidden by our assumptions. Thus this sublattice is isometric to $\qf{1,\ve,2+\overline{\ve}}$.
		
		Since we fixed the vectors representing $1$ and $\ve$, we now arrived at the conclusion that the sublattice generated by all the four vectors representing $1$, $\ve$, $2+\ve$ and $2+\overline{\ve}$ together must be of the form $\qf{1,\ve} \perp J$ where $J\simeq\bigl\langle\begin{smallmatrix} 2+\ve & a \\ a & 2+\overline{\ve} \end{smallmatrix}\bigr\rangle$, just as in Lemma \ref{le:2+eps2+bareps}. By this lemma, since we assume that $M_i$ neither is universal nor contains $\qf{1,\ve,2}$, there remain only six possibilities for $a$. For five of them, Conjecture \ref{co:5lattices} or Lemma \ref{le:boundIII} directly yield the implication \enquote{$\alpha\not\to M_i$ implies $\alpha\in R$}. For the sixth, with $a=1-\sqrt3$, we know that the statement holds for the conjugate lattice; and since $R$ is closed under conjugation, the desired implication follows.
		
		Now that we proved the claim, the rest is easy. Assume $\alpha \not\to M_i$. Then, using Proposition \ref{pr:regularternary} for its characterisation of elements represented by $\qf{1,\ve,2}$ yields:
		\[
		\alpha \equiv \pm(1+2\sqrt{3}) \qquad\text{ or } \qquad \alpha \in R.
		\]
		So it remains to check that if these conditions apply to $\alpha$, then they cannot apply to $\ve\alpha$ as well. This is straightforward. We see that $R \cap \ve R = \emptyset$ and that if $\alpha\equiv \pm(1+2\sqrt{3})\pmod4$, then $\ve\alpha \equiv \pm(2+\sqrt3)(1+2\sqrt3)\equiv \pm\sqrt3 \pmod4$. The remaining implications to be checked are \enquote{$\alpha \in R$ implies $\ve\alpha\not\equiv \pm(1+2\sqrt3)$} and \enquote{$\alpha\equiv \pm(1+2\sqrt3)$ implies $\ve\alpha \not\in R$}; they are in fact equivalent and one easily checks one of them.
	\end{proof}
	
	With this preparation, the rest of the proof is short and elegant.
	
	\begin{proof}[Proof of Proposition \ref{pr:sqrt3main}]
		Let $(L,Q)$ be a lattice representing all of $1,\ve,2+\ve,2+\overline{\ve}, 1+2\ve, 1+2\overline{\ve}$.
		
		If $L$ contains $\qf{1,\ve,1+\ve}$ or its conjugate $\qf{1,\ve,1+\overline{\ve}}$, then by Conjecture \ref{co:sqrt3ternary} (or by Lemma \ref{le:boundHuge} if we do not assume the conjecture but instead study $\alpha$ with bounded norm) this ternary sublattice already represents all elements except for two elements of norm $13$ and two elements of norm $37$. Since all elements of norm $13$ are represented by $L$ by assumption, the only possible non-represented elements have norm $37$ and the proposition is proved.
		
		So in the rest of the proof we assume $L$ to contain neither $\qf{1,\ve,1+\ve}$ nor $\qf{1,\ve,1+\overline{\ve}}$. Observe that then the same holds for the scaled version $(L,\ve Q)$ as well. The crucial point now is that both the lattice $(L,Q)$ and its scaled version $(L,\ve Q)$ represent $1,\ve,2+\ve$ and $2+\overline{\ve}$. Thus they both satisfy the assumptions of Lemma \ref{le:technical}.
		
		Hence, if we assume Conjecture \ref{co:5lattices} or $\NN(\alpha)\leq \boundIII{}$, then we get either $\alpha \to (L,Q)$, or $\ve\alpha \to (L,\ve Q)$, which is of course equivalent to $\alpha \to (L,Q)$ as well.
		
		This concludes the proof -- and we see that in fact, if $L$ contains neither of the two ternary sublattices, it is actually universal (under the Conjectures or up to norm $\boundIII{}$, of course).
	\end{proof}
	
	By proving Proposition \ref{pr:sqrt3main}, the proof of Theorem \ref{th:sqrt3} is concluded.

	\section{Diagonal criteria for \texorpdfstring{$\qq{D}$}{Q(√D)}, \texorpdfstring{$D=2,{3},5$}{D=2, 3, 5}} \label{se:diagExplicit}
	
	Let us now briefly discuss the classical universality criterion for $\qq{5}$ and then the diagonal criteria for the first three quadratic fields.
	
	Recall that  Lee computed the universality criterion set for classical lattices over $\qq{5}$. Namely, \cite[Thm.~3.5]{Le}, which we already stated in the Introduction as the Norm-45-Theorem, claims that a classical quadratic lattice over $\qq{5}$ is universal if and only if it represents all elements of
	\[
	\Ccl_{\qq{5}}=\Bigl\{1,2,\frac{5+\sqrt5}{2}, \frac{7\pm\sqrt5}{2},2\cdot\frac{5+\sqrt5}{2}, 3\cdot\frac{5+\sqrt5}{2}\Bigr\}.
	\]
	
	We are able to strengthen this result as follows.
	
	\begin{theorem} \label{th:sq5unique}
		The set $\Ccl_{\qq{5}}$ given by Lee is minimal (i.e.\ it is indeed $\Ccl_{\qq{5}}$ in the sense of \cite{KKR}): For every $\alpha \in \Ccl_{\qq{5}}$, there exists a classical quadratic lattice $L_{\alpha}$ which represents all other elements of $\Z\bigl[\frac{1+\sqrt5}{2}\bigr]^{+}$ but not $\alpha\U_{\qq{5}}^2$.
	\end{theorem}
	Note that Lee did not construct the lattices $L_{\alpha}$, and neither shall we. Their existence follows from the fact that Lee found lattices which have them as their truants, so one can use Proposition \ref{pr:KKRcritical}.
	\begin{proof}
		By Lee's result, this set is a criterion set. Therefore it contains all critical elements. It remains to prove that all its elements are indeed critical.
		
		By Proposition \ref{pr:KKRcritical}, it is enough to find a lattice with truant $\alpha$ for every $\alpha\in\Ccl_{\qq{5}}$. They can be found in Lee's computations (but it is easy to check the truants independently): $\truant\bigl(\{\vectorstyle{0}\}\bigr)=1$; $\truant\bigl(\qf{1}\bigr)=2$; $\truant\bigl(\qf{1,2}\bigr)=\frac{5+\sqrt5}{2}$; $\truant\bigl(\qf{1,1}\bigr)$ is not unique and takes two values $\frac{7\pm\sqrt5}{2}$; $\truant\bigl(\qf{1,2,\frac{5+\sqrt5}{2}}\bigr)=2\cdot\frac{5+\sqrt5}{2}$; $\truant(M)=3\cdot\frac{5+\sqrt5}{2}$ for $M \simeq \qf{1,2}\perp \bigl\langle\begin{smallmatrix} (5+\sqrt5)/2 & (1+\sqrt5)/2 \\ (1+\sqrt5)/2 & 5+\sqrt5 \end{smallmatrix}\bigr\rangle$.
	\end{proof}

	When one knows $\Ccl_K$, one fully understands universality of classical quadratic forms over $K$. Therefore, one can use it to compute the criterion set for diagonal forms.
	
	\begin{theorem} \label{th:diagonal}
		We have
		\[
		\Cdiag_{\qq{5}} = \Bigl\{1,2,\frac{5+\sqrt5}{2}, \frac{7\pm\sqrt5}{2},2\cdot\frac{5+\sqrt5}{2}\Bigr\}.
		\]
		
		Assuming universality of $E_{4,(5)}^{(1)}\simeq\qf{1,2+\sqrt2,3,3(2+\sqrt2)}$, 
		$E_{4,(5)}^{(19)}\simeq\qf{1,2+\sqrt2,3,2(2+\sqrt2)}$ and
		$E_{4,(5)}^{(41)}\simeq\qf{1,2+\sqrt2,3,3+\sqrt2}$, we also get
		\[
		\Cdiag_{\qq{2}} = \bigl\{1,2+\sqrt2,3,3(2+\sqrt2)\bigr\}.
		\]
		In both cases, every universal diagonal form contains a ternary or quaternary universal diagonal subform.
	\end{theorem}
	Although we know (from \cite{Le} and from Section \ref{se:sqrt2}) the full lists of classical escalators, this is of no direct use for the proof, since they do not help us to produce pseudoescalators. We will compute the pseudoescalators independently, based on the explicit description from Proposition \ref{pr:diagonalpractical}, and \enquote{only} use the knowledge of $\Ccl_{\qq{5}}$ as an effective criterion for proving universality. (We could do the same with the conditional knowledge of $\Ccl_{\qq{2}}$, but we do not want to assume the full Conjecture \ref{co:13lattices}, and so we compute the diagonal criterion for $\qq{2}$ independently on the main results of Section \ref{se:sqrt2}.)
	\begin{proof}[Proof of Theorem \ref{th:diagonal}]
		We compute pseudoescalators using Proposition \ref{pr:diagonalpractical}. For $\qq{5}$, we will use the notation $\varphi = \frac{1+\sqrt5}{2}$. The truant of $\{\vectorstyle{0}\}$ is $1$, this has no decomposition, so the only pseudoescalation is $\qf{1}$ with truant $2$. This decomposes as $1+1$ or $2$, so the next pseudoescalators are $\qf{1,1}$ and $\qf{1,2}$. The former has truants $3+\varphi$ and $3+\overline{\varphi}$. This already means that both of them belong to $\Cdiag$; but we have to choose an order $\rho$, for example $3+\varphi <_\rho 3+\overline{\varphi}$, to continue with the escalation. One has $0 \prec \beta \preceq 3+\varphi$ precisely for $\beta = 1$, $2$, $1+\varphi$, $2+\varphi$ and $3+\varphi$; we can omit $1+\varphi = \varphi^2 \in 1\cdot\U_{\qq{5}}^2$. These are all squarefree elements, so we do not have to add any other, and the pseudoescalations are $\qf{1,1,\beta}$ for $\beta = 1,2,2+\varphi, 3+\varphi$. All but the last one are universal (by \cite{CKR} or by application of $\Ccl_{\qq{5}}$). The last one, $\qf{1,1,3+\varphi}$, has truant $3+\overline{\varphi}$ as expected; symmetrically, its decompositions contain $\beta = 1$, $2$, $1+\overline{\varphi}$, $2+\overline{\varphi}$ and $3+\overline{\varphi}$, and $1+\overline{\varphi}$ can again be omitted. The four resulting pseudoescalators $\qf{1,1,3+\varphi,\beta}$ for $\beta = 1$, $2$, $2+\overline{\varphi}$, $3+\overline{\varphi}$ are all universal, since they represent all of $\Ccl_{\qq{5}}$. In the other branch, $\qf{1,2}$ has truant $2+\varphi$, and since we can again replace $1+\varphi$ by $1$, there are only two pseudoescalators, namely $\qf{1,2,1}$ which is universal, and $\qf{1,2,2+\varphi}$ with truant $2(2+\varphi)$. Studying the decompositions (and omitting duplicities such as $3+\sqrt5 = 2\varphi^2$), one finally arrives at the pseudoescalators $\qf{1,2,2+\varphi,\beta}$ for $\beta = 1$, $2$, $2+\varphi$, $3+\varphi$, $3+\overline{\varphi}$, $2(2+\varphi)$. They can all be checked to be universal. Once the escalation procedure is finished, we find $\Cdiag$ as the collection of all truants which we encountered.
		
		For $\qq{2}$, the process is the same, and there are actually fewer pseudoescalators: $\{\vectorstyle{0}\}$ with truant $1$, then $\qf{1}$ with truant $2+\sqrt2$, then $\qf{1,2+\sqrt2}$ with truant $3$. The decompositions of $3$ contain only $1,2$ and $3$; although $2$ is not squarefree, division by the nontrivial square gives $1$, which was already included in our list anyway, so there are only three pseudoescalators: $\qf{1,2+\sqrt2,1}$, which is universal, and our old acquaintances $E_3^{(6)}\simeq \qf{1,2+\sqrt2,2}$ and $E_3^{(5)}\simeq \qf{1,2+\sqrt2,3}$, both with truant $3(2+\sqrt2)$. Studying the decompositions, omitting $(1+\sqrt2)^2$ and replacing $5+3\sqrt2$ by $3-\sqrt2$, we conclude that in both cases $i=2,3$, the pseudoescalators are given by $\qf{1,2+\sqrt2,i,\beta}$ where $\beta = 1$, $2+\sqrt2$, $3\pm\sqrt2$, $2(2+\sqrt2)$, $3(2+\sqrt2)$. They are all universal -- if one assumes a selected part of Conjecture \ref{co:13lattices}.
	\end{proof}
	
	As for $\qq{3}$, the diagonal criterion is already included in our main results (Conjecture \ref{co:sqrt3} and Theorem \ref{th:sqrt3}). Note that, since our aim was to prove the equality of $\Ccl_{\qq{3}}$ with $\Cdiag_{\qq{3}}$, there was no need perform the pseudoescalation in this case. The reason is that, in contrast with the proof of Theorem \ref{th:diagonal}, we did not need to show that some element of $\alpha\in\Ccl$ does not belong into $\Cdiag$ by explicitly listing all pseudoescalations and checking that none of them has $\alpha$ as truant. Rather, for every $\alpha\in\Ccl$, it was enough to find one diagonal form with truant $\alpha$, and it was of no consequence whether this form corresponds to an escalator, or is a pseudoescalator, or neither of these. (Also, thanks to symmetry it was in fact enough to find three and not ten such diagonal forms, see the proof of Theorem \ref{th:sqrt3}.)
	
	One potential reason for going through the pseudoescalation procedure anyway would be to prove the diagonal part of Theorem \ref{th:sqrt3} under weaker assumptions. However, at least Conjecture \ref{co:sqrt3ternary} would still need to be assumed. And let us stress once again that one cannot obtain the sufficient conjecture just by omitting all non-diagonal forms from Conjecture \ref{co:5lattices} without any computation; it is quite possible that some of them occurs as a sublattice of a diagonal pseudoescalator.

	\section{Maximal critical element} \label{se:amax}
	
	We have observed an interesting phenomenon regarding the largest element of $\CC_K$. It is analogous to the following well-known \cite[Thm.~2.3 and 2.8]{Moon} facts: \emph{If a quadratic form (classical quadratic form, resp.) over $\Q$ represents all numbers $<290$ ($<15$, resp.), then it represents all numbers $>290$ ($>15$, resp.).} The same phenomenon was recently observed for coprime-universal quadratic forms over $\Q$, see \cite[Cor.~1.7]{BC}.
	
	\begin{proposition} \label{pr:non-con}
		Let $K = \qq{D}$ for $D=2,3,5$, and assume the validity of Conjecture \ref{co:13lattices} if $D=2$ and the validity of Conjectures \ref{co:sqrt3ternary} and \ref{co:5lattices} if $D=3$. 
		
		Let $\amax$ be an element of maximal norm in $\Ccl_K$. Then:
		\begin{enumerate}
			\item Any lattice representing $\CC_K \setminus \{\amax\}$ represents all of $\OKPlus$ except possibly $\amax$. \label{it:noncon1}
			\item More generally, any lattice representing all elements of strictly smaller norm than $\NN(\amax)$ also represents all elements of strictly bigger norm. \label{it:noncon2}
		\end{enumerate}
		The analogous statement for $\Cdiag_K$ holds if $D=3$ but fails for $D=2$ and $D=5$.
	\end{proposition}
	\begin{proof}
		Note that the proposition is equivalent to the same statement where \enquote{lattice} is replaced by \enquote{escalator}. Indeed, one implication is trivial, and the other is due to the fact that every lattice contains an escalator with the same truant, see Proposition \ref{pr:subescalators}.
		
		In $\qq{5}$, we have $\amax=3\cdot\frac{5+\sqrt5}{2}$, and $\Ccl_{\qq{5}}$ contains no other element of the same norm. Thus a lattice satisfying the conditions in \ref*{it:noncon1} or \ref*{it:noncon2} is either universal and there is nothing to prove, or it has $\amax$ as the truant. By \cite[Prop.~3.2]{Le}, there are only three escalators with this truant, and all of them represent all other elements.
		
		Assuming Conjecture \ref{co:13lattices}, in $\qq{2}$ we have two maximal elements, $3(3\pm\sqrt2)$. We have seen that for each of them there exists only one escalator with this truant, namely $E_{4,(5)}^{(3)}$ and $E_{4,(5)}^{(2)}$, respectively, and it indeed represents everything except the (unique) truant. This proves \ref*{it:noncon1} as well as \ref*{it:noncon2}.
		
		For $\qq{3}$, the statement is just a rephrasing of Proposition \ref{pr:sqrt3main}.
		
		We conclude with a brief discussion of the diagonal case. For $D=3$, we (conjecturally) have $\Cdiag_{\qq{3}}=\Ccl_{\qq{3}}$, so if the statement holds for all classical lattices, then it trivially holds for diagonal forms. For $D=2$, we (conjecturally) have $\Cdiag_{\qq{2}} = \bigl\{1,2+\sqrt2,3,3(2+\sqrt2)\bigr\}$ by Theorem \ref{th:diagonal}; we know (and can easily check) that $E_3^{(5)}\simeq \qf{1,2+\sqrt2,3}$ has truant $3(2+\sqrt2)$, but there are many other non-represented elements, for example $3(3+\sqrt2)$. For $D=5$, we know that $\Cdiag_{\qq{5}} = \bigl\{1,2,\frac{5+\sqrt5}{2}, \frac{7\pm\sqrt5}{2},2\cdot\frac{5+\sqrt5}{2}\bigr\}$ thanks to Theorem \ref{th:diagonal}; again, the lattice $\qf{1,2,\frac{5+\sqrt5}{2}}$ has truant $2\cdot \frac{5+\sqrt5}{2}$, but it fails to represent many other elements as well, e.g.\ $3\cdot \frac{5+\sqrt5}{2}$.
	\end{proof}
	
	As we already mentioned, a very similar situation occurs for $\Ccl_{\Q} = \Cdiag_{\Q}$ and $\CC_{\Q}$ as well as for $\CC_{S^{(p)}}$ and $\Ccl_{S^{(p)}}$ where $p$ is a prime and $S^{(p)}=\{\text{$n$ coprime to $p$}\}$, see \cite[Cor.~1.7]{BC}. Therefore we ask the following question:
	
	\begin{question} \label{qu:estion}
		How often does the situation analogous to Proposition \ref{pr:non-con}\ref*{it:noncon1} occur for a totally real number field $K$? Is there some deeper explanation why such a statement should often be true? What about $\CC_S$ or $\Ccl_S$ for suitable infinite sets $S$?
	\end{question}
	
	In a follow-up paper \cite{Kr}, we will present strong computational evidence that $\Ccl_{\qq{21}}=\{1,\ve,2,2\ve\}$, where $\ve=\frac{5+\sqrt{21}}{2}$. Assuming this is indeed the full criterion set, this field gives one example of a negative answer to the Question: The lattice $\qf{1,2,\ve}$ represents all of $\Ccl_{\qq{21}}$ except for $\amax=2\ve$, but there are many other non-represented elements, for example $1+2\ve=6+\sqrt{21}$.
	
	One can argue that the statement is true for $D=2$, $3$, $5$ simply because it is true for a few escalators; and since there are always only finitely many escalators, it is possible that in all the observed cases it happens purely by chance.
	
	Let us make one further comment regarding Proposition \ref{pr:non-con}: For $\Z[\sqrt2]$, we really have to assume Conjecture \ref{co:13lattices} (the one about $13$ specific lattices) to obtain it. Assuming directly the main Conjecture \ref{co:mainsqrt2} (the exact value of $\Ccl_{\qq{2}}$) is not enough; and the reason \emph{why} is quite interesting, and connected to Question \ref{qu:estion}. 
	
	Imagine that we \emph{know} that $3 (3\pm\sqrt2)$ are indeed the largest elements of $\Ccl_{\qq{2}}$. However, assume that the lattice $E_{4,(5)}^{(2)}$ with truant $3(3-\sqrt2)$ fails to represent some huge element $\beta$. To obtain the necessary contradiction, one would need to show that at least one of the escalations of $E_{4,(5)}^{(2)}$ will still fail to represent $\beta$; and this sound plausible but by no means certain. How do we even know that this cannot happen for some $\beta = 3(3-\sqrt2)\omega^2$?

	\bigskip
	
	Finally, let us say a few words about invariants connected with the universality criterion sets. Describing the criterion set explicitly for any family of fields seems very hard. However, there are at least three quantities connected with criteria for which it would be interesting to obtain good upper and lower bounds. These are:
	\begin{enumerate}
		\item The size of the criterion set, $\#\CC_K$,
		\item the largest norm in the criterion set, $\max\{\NN(\alpha) \mid \alpha \in \CC_K\}$,
		\item \enquote{the largest escalator}, $\min\{d \mid \text{every escalator of rank $d$ is universal}\}$.
	\end{enumerate}
	Note that the exact knowledge of $\#\CC_K$ or the knowledge of any upper bound on the second quantity yields an \enquote{effective} algorithm for computing the criterion set, since the method of escalation allows us to compute any finite part of $\CC_K$. While in \cite{KKR} it was remarked that $\#\CC_K \geq \Cdiag_K \geq m_K^{\diag}\geq m_K$ where $m_K$ is the minimal rank of a universal quadratic lattice over $K$, no upper bound on $\#\CC_K$ is known.
	
	Regarding the last quantity, let us just mention that it might be seen as the integral analogue of the field invariant $u(K)$ (one variant of which can be defined roughly as the lowest dimension for which every positive definite quadratic space over $K$ is universal). Of course there exist non-universal integral quadratic lattices of any dimension, hence one might argue that restricting the attention to escalators (which are sincerely trying to become universal in each step) is really the best possible analogue. The analogy is clearly visible from the formulation that $u(K)$ is the lowest number $d$ such that every universal positive definite quadratic space contains a universal quadratic subspace of dimension at most $d$.
	
	\smallskip
	
	The computations were carried out in Magma \cite{magma}; the programs are available upon request. The authors thank Vítězslav Kala for several helpful discussions. The second author was supported by the project PRIMUS/25/SCI/008 from Charles University.

\end{document}